\documentclass[11 pt]{amsart}
\usepackage{amssymb}
\usepackage{amsmath}
\usepackage{amsfonts}
\usepackage{graphicx}
\usepackage{amsthm}
\usepackage{enumerate}
\usepackage[mathscr]{eucal}
\usepackage{verbatim}
\newtheorem{theorem}{Theorem}[section]
\newtheorem{lemma}[theorem]{Lemma}
\newtheorem{observation}[theorem]{Observation}
\newtheorem{proposition}[theorem]{Proposition}

\newtheorem*{definition}{Definition}

\newtheorem{hyps}[theorem]{Hypotheses}
\theoremstyle{plain}

\theoremstyle{remark}

\newcommand{\ci}[1]{_{{}_{\scriptstyle{#1}}}}
\newcommand{\Be}{\begin{equation}}
\newcommand{\Ee}{\end{equation}}

\def\sD{{\mathscr {D}}}
\def\sH{{\mathscr {H}}}
\def\sG{{\mathscr {G}}}
\def\vol{{\text{\rm vol}}}
\def\thiz{{\theta_{i,0}}}
\def\thio{{\theta_{i,1}}}

\def\siz{{s_{i,0}}}
\def\sio{{s_{i,1}}}
\def\qiz{{q_{i,0}}}
\def\qio{{q_{i,1}}}
\def\vars{\varsigma}
\def\ic{{{\text{\rm i}}}}

\def\intslash{\rlap{\kern  .32em $\mspace {.5mu}\backslash$ }\int}
\def\qsl{{\rlap{\kern  .32em $\mspace {.5mu}\backslash$ }\int_{Q_x}}}
\def\Re{\operatorname{Re\,}}

\def\vpi{\pi}
\def\vth{\vartheta}

\def\emph#1{{\it #1 }}

\def\th{\theta}

\def\prel{{\text{\it  prel}}}

\def\cf{{\it cf}}

\def\inn#1#2{\langle#1,#2\rangle}

\def\noi{\noindent}

\def\meas{{\text{\rm meas}}}

\def\lc{\lesssim}
\def\gc{\gtrsim}

\def\ga{\gamma}

\def\eps{\varepsilon}

\def\ka{\kappa}

\def\la{\lambda}
\def\La{\Lambda}
\def\sig{\sigma}

\def\fC{{\mathfrak {C}}}

\def\fa{{\mathfrak {a}}}

\def\fc{{\mathfrak {c}}}

\def\bbN{{\mathbb {N}}}

\def\bbR{{\mathbb {R}}}

\def\bbZ{{\mathbb {Z}}}

\def\cB{{\mathcal {B}}}
\def\cC{{\mathcal {C}}}

\def\cE{{\mathcal {E}}}
\def\cF{{\mathcal {F}}}

\def\cH{{\mathcal {H}}}
\def\cI{{\mathcal {I}}}
\def\cJ{{\mathcal {J}}}

\def\cL{{\mathcal {L}}}
\def\cM{{\mathcal {M}}}

\def\cP{{\mathcal {P}}}

\def\cR{{\mathcal {R}}}

\def\cZ{{\mathcal {Z}}}

\def\C{{\hbox{\bf C}}}

\def\f{{\overline{f}}}
 at 10 true pt

\def\be#1{\begin{equation}\label{#1}}
\def\ee{\end{equation}}
\def\bas{\begin{align*}}
\def\eas{\end{align*}}
\def\bi{\begin{itemize}}
\def\ei{\end{itemize}}

\def\eps{\varepsilon}
\def\emph#1{{\it #1}}
\def\textbf#1{{\bf #1}}
\def\intslash{\rlap{\kern  .32em $\mspace {.5mu}\backslash$ }\int}
\def\qsl{{\rlap{\kern  .32em $\mspace {.5mu}\backslash$ }\int_{Q_x}}}

\begin{document}


\title
[An endpoint estimate with affine arclength measure]
{Restriction of Fourier transforms to curves:
\\ An endpoint estimate \\ with affine  arclength measure}

\author[]
{Jong-Guk Bak \ \ Daniel M. Oberlin \  \  Andreas Seeger}

\address {J. Bak \\ Department of Mathematics\\ Pohang University of Science and Technology
\\
Pohang 790-784, Korea}
\email{bak@postech.ac.kr}

\address
{D. M.  Oberlin \\
Department of Mathematics \\ Florida State University \\
 Tallahassee, FL 32306, USA}
\email{oberlin@math.fsu.edu}

\address{A. Seeger   \\
Department of Mathematics\\ University of Wisconsin-Madison\\Madison,
WI 53706, USA}
\email{seeger@math.wisc.edu}



\thanks{Supported in part by National Research Foundation 
grant 2010-0024861
from the Ministry of Education, Science and Technology of Korea,
and by 
National Science Foundation  grants DMS-0552041 and    DMS-0652890.
}

\begin{abstract}
Consider the Fourier restriction operators
associated to  curves in $\mathbb R^d$, $d\ge 3$.
We prove 
 for various classes of curves
the endpoint restricted strong
type estimate with respect to affine arclength measure on the curve.
An essential ingredient is an  interpolation result  for
multilinear operators with symmetries 
acting on  sequences  of vector-valued functions.
\end{abstract}

\maketitle

\section{Introduction}\label{intro}
Let $t\mapsto \gamma(t)$ define a  curve in $\bbR^d$,
defined for $t$ in a parameter interval $I$. We shall assume that
$\gamma$ is at
least of class $C^d$ on $I$.

In this paper we investigate the mapping properties of the Fourier 
restriction operator associated to the curve, given for Schwartz functions on
$\bbR^d$  by
$$\cR f (t) = \widehat f(\gamma(t));$$
here the Fourier transform is defined by $\widehat f(\xi)=\int f(y) e^{-\ic \inn{y}{\xi} }d\xi$. 
$\cR f$ will be measured in Lebesgue spaces 
$L^q(I; d\la)$
where  $d\la= w(t)dt$ is {\it affine arclength measure} 
with  weight 
\begin{equation}\label{weight}w(t) = |\tau(t)|^{\frac{2}{d^2+d}} \quad \text{where }
\tau(t)= \det (\gamma'(t),\dots, \gamma^{(d)}(t)).
\end{equation}
The relevance of affine arclength measure for harmonic analysis  
 has been discussed in   \cite{D2} and \cite{obcurv}. 
There is an invariance under change of variables and 
reparametrizations.  Fourier restriction theorems  for the  
case of  \lq nondegenerate\rq \  curves (with nonvanishing $\tau$)
are  supposed to extend to large classes of \lq degenerate\rq \ curves  
when arclength measure is replaced 
by affine arclength measure, 
with uniform 
constants in the estimates. Finally  
 the  choice of
affine arclength measure is optimal
up to multiplicative constants, 
in a sense made precise in the next section.

For nondegenerate curves affine arclength measure is 
comparable to the standard arclength measure on any compact interval.
Note that for the model case $(t,t^2,\dots, t^d)$ the weight $w$
is constant (equal to $(d!)^{\frac{2}{d^2+d}})$.
The sharp $L^p\to L^\sigma$ estimates for this case have been obtained by Zygmund \cite{Z} and H\"ormander \cite{H} in the case $d=2$ and
by Drury \cite{D} in
higher dimensions. Namely, one gets $L^p(\bbR^d)\to L^\sigma(\bbR)$ boundedness for
$1<p<p_d:=\frac{d^2+d+2}{d^2+d}$, $p'=\sigma \frac{d(d+1)}2$. A nonisotropic
scaling reveals that for a global estimate this relation between $p$ and $\sigma$ is necessary in this case. Moreover, it follows from a result by Arkhipov, Chubarikov and Karatsuba  \cite{act}
that the given  range of $p$  is optimal. One can ask for weaker
estimates at the endpoint $p_d$ which imply the $L^p\to L^\sigma$ estimates by 
interpolation. The iterative method  by Drury fails to give information
at the endpoint. In two dimensions,
Beckner, Carbery, Semmes and Soria \cite{bcss} have
shown that even the restricted weak type estimate
fails  at the endpoint $p_2=4/3$. However, in \cite{bos1} the authors proved
for the nondegenerate  model case
that in dimensions $d\ge 3$ the Fourier restriction operator is of
{\it restricted strong type} $(p_d,p_d)$, i.e. maps the Lorentz space $L^{p_d,1}(\bbR^d) $  to $L^{p_d}(\bbR, dt)$.
This result is optimal with respect to the secondary Lorentz exponents.

It is natural to ask whether for more general classes of curves
the endpoint inequality
\begin{equation} \label{gammaendpt}
\Big(\int_I |\widehat f\circ \gamma|^{p_d} d\la \Big)^{1/p_d} \lc
\|f\|_{L^{p_d,1}(\bbR^d)},\quad
p_d=\frac{d^2+d+2}{d^2+d},
\end{equation}
holds true 
with affine arclength measure $d\la$.
This estimate of course implies 
the best possible $L^p(d\la)\to L^q$ bounds which for some  classes of curves 
 were proved in the first two papers of this series \cite{bos1}, \cite{bos2},
 building on earlier work by Drury and Marshall 
\cite{DM1}, \cite{DM2}. See also  the very  recent work by 
M\"uller and Dendrinos \cite{DeM} for further extensions. In two dimensions 
the endpoint bound fails  and sharp Lebesgue space estimate can be found in 
\cite{sj}, \cite{ob}.

Here we prove  
\eqref{gammaendpt} for two classes
of  curves. We first consider the case of ``monomial'' curves of the form
\begin{equation}\label{monomial}
t\mapsto \gamma_a(t)=(t^{a_1}, t^{a_2},\dots, t^{a_d}), \quad 0<t<\infty
\end{equation}
where $a=(a_1,\dots, a_d)$ are arbitrary real numbers, $d\ge 3$. 

\begin{theorem} \label{powerthm}
Let $d\ge 3$ and let  $w_a dt$ denote the affine arclength measure 
 for the curve
\eqref{monomial}.
Then there is $C(d)<\infty$ so that for
all $f\in L^{p_d ,1}(\bbR^d)$
\begin{equation} \label{affrestr2}
\Big(\int_0^\infty
 |\widehat f (\gamma_a(t))|^{p_d} w_a(t) dt\Big)^{1/p_d} \le
C(d)\|f\|_{L^{p_d ,1}(\bbR^d)}.
\end{equation}
\end{theorem}
Note that the constant in \eqref{affrestr2} is universal in the sense that it does not depend on $a_1,\dots, a_d$.

A similar result holds for `simple' polynomial curves in $\bbR^d$, $d\ge 3$,

\begin{equation} \label{polcurve}
\Gamma_b(t)= \Big(t, {t^2\over 2!},\cdots,
{t^{d-1}\over (d-1)!} ,P_b(t)\Big), \quad t\in \bbR,
\end{equation}
 where $P_b$  is an arbitrary
polynomial of degree $N\ge 0$, with the coefficients $(b_0,
\cdots, b_N)=b\in \bbR^{N+1}$, that is, 
$P_b(t) = \sum_{j=0}^N b_j t^j .$
Note that the affine arclength measure in this case is
given by $W_b(t) dt$ where $W_b(t)= |P_b^{(d)}(t)|^{\frac{2}{d^2+d}}.$
Then we have
\begin{theorem} \label{polthm}
There is $C(N)<\infty$ so that for
all $f\in L^{p_d ,1}(\bbR^d)$,  $b\in \bbR^{N+1}$,
\begin{equation} \label{affrestrpol}
\Big(\int_0^\infty
 |\widehat f (\Gamma_b(t))|^{p_d} W_b(t) dt\Big)^{1/p_d} \le
C(N)\|f\|_{L^{p_d ,1}(\bbR^d)}.
\end{equation}
\end{theorem}
It would be interesting to prove a similar theorem for general polynomial curves $(P_1(t), \dots, P_d(t))$, with a bound depending only on the highest degree.
 However, currently we do not even know
the sharp $L^p\to L^q(w)$ bounds in the optimal range
$p\in [1, \frac{d^2+d+2}{d^2+d})$. For the smaller range
$1\le p<\frac{d^2+2d}{d^2+2d-2}$ (corresponding to the range in Christ's paper 
\cite{Ch} for the nondegenerate case), such universal $L^p\to L^q(w)$ bounds 
have been recently proved by Dendrinos and Wright \cite{DeW}. Their result 
can be slightly extended by combining an argument by Drury \cite{D2}
 with estimates by Stovall \cite{sto} on averaging operators, 
see  \S\ref{Drury-estimate}.

\subsection*{An interpolation theorem}
As in  previous papers on restriction theorems for curves 
the results rely on the analysis 
of multilinear operators with a high degree of symmetry. 
In \cite{bos1} the operators acted on $n$-tuples of functions in 
Lebesgue or  Lorentz  spaces, and it was important to use an interpolation procedure introduced by Christ in \cite{Ch} (\cf.  also \cite{janson}, \cite{grkal} for related results). In the presence of weights one is  led  to consider 
interpolation  
results for $n$-linear operators acting on  products of  $\ell_s^p(X)$ spaces
and which have values in a Lorentz space;  here $X$ is a 
quasi-normed space and $\ell^p_s(X)$ is the space of $X$ valued sequences
$\{f_k\}_{k\in \bbZ}$  for which 
$(\sum_{k\in \bbZ }2^{ksp}\|f_k\|_X^p)^{1/p}<\infty$.  
For the relevance to the restriction problem see  also the remarks following the statement of  Theorem \ref{interpolationtheorem} below.

We recall some terminology from interpolation theory. A quasi-norm on a  vector space has the same properties as a norm except that the triangle inequality is weakened to 
$\|x+y\|\le C(\|x\|+\|y\|)$ for some constant $C$. Let $0<r\le 1$. 
The topology generated by the balls defined by this norm is called {\it $r$-convex} if there is a constant $C_1$ so that
\Be\label{triangleineq}\Big\|\sum_{i=1}^n x_i\Big\|_X\le C_1 (\sum_{i=1}^n \|x_i\|_X^r)^{1/r}
\Ee  holds 
for any finite sums of elements in $X$.
The Aoki-Rolewicz theorem states that every quasi-normed space is $r$-convex for some $r>0$ (see also \S3.10 in \cite{BL} for a generalization).
Obviously any normed space is $1$-convex.
Hunt \cite{hu} 
 showed that Lorentz spaces $L^{pq}$ are $r$-convex for $r< \min\{1,p,q\}$
and they are normable for $p,q>1$. 
The Lorentz space $L^{r,\infty}$ is $r$-convex for $0<r<1$; this is a result 
by Kalton 
\cite{kalton} and by   Stein, Taibleson and Weiss  \cite{stw}. This fact plays 
a role in the proof of sharp endpoint theorems, in \cite{bos1} as well as in the present paper.

The Lions-Peetre interpolation theory can be extended to quasi-normed 
spaces (see \S3.11 of \cite{BL}). Here one works with couples $\overline X=(X_0,X_1)$ of 
compatible quasi-normed spaces, i.e.  both $X_0$ and $X_1$ are continuously embedded in some topological vector space. We shall use both 
the $K$-functional  defined on $X_0+X_1$, given by $K(t,f;\overline X)=
\inf_{f=f_0+f_1} [\|f_0\|_{X_0}+t\|f_1\|_{X_1}]$ and the $J$-functional defined on $X_0\cap X_1$ by $J(t,f,\overline X)= \max\{\|f\|_{X_0}, t\|f\|_{X_1}\}$.
For $0<\theta<1$, $0<q<\infty$ 
 the interpolation space $\overline X_{\theta,q}$ 
is the space of $f\in X_0+X_1$
for which 
$\|f\|_{\overline X_{\theta,q}}= (\sum_{l\in \bbZ} [2^{-l\theta} K(2^l, f;\overline X)]^q)^{1/q} $ is finite. 
Similarly one defines $\overline X_{\theta,\infty}$ with quasi-norm 
$\|f\|_{\overline X_{\theta,\infty}}=
 \sup_{l\in \bbZ} 2^{-l\theta} K(2^l, f;\overline X)$. The space $X_0\cap X_1$ is dense 
in $\overline X_{\theta,q}$ but not necessarily in
$\overline X_{\theta,\infty}$;
 the closure of  $X_0\cap X_1$ in 
$\overline X_{\theta,\infty}$ is denoted by
$\overline X_{\theta,\infty}^0$ and consists of all 
$f\in \overline X_{\theta,\infty}$ for which $2^{-l\theta} K(2^l,f:\overline X)$ tends to $0$ as $l\to \pm\infty$.
An   equivalent
norm on $\overline X_{\theta,q}$ is given by 
$\|f\|_{\overline X_{\theta,q;J}}=  \inf (\sum_{l\in \bbZ} [2^{-l\theta} J(2^l, u_l;\overline X)]^q)^{1/q} $, where the infimum is taken over all representations 
$f=\sum_l u_l$,  $u_l\in X_0\cap X_1$, with convergence in $X_0+X_1$ (see the equivalence theorem 3.11.3 in \cite{BL}).

For the formulation and  proofs of interpolation results for multilinear operators with symmetries it is convenient to use the notion of a {\it doubly stochastic $n\times n$ matrix}, i.e. a matrix $A=(a_{ij})_{i,j=1,\dots, n}$ for which
$a_{ij}\in [0,1]$, $i,j=1,\dots, n$, $\sum_{j=1}^n a_{ij}=1$, $i=1,\dots, n$ and
$\sum_{i=1}^n a_{ij}=1$, $j=1,\dots, n.$
Doubly stochastic matrices arise naturally in the interpolation of  operators
with symmetries under permutations; this is because of Birkhoff's theorem
 (\cite{birk}, \cite{marcusminc})
which states that
the set of all doubly stochastic matrices (also called the Birkhoff polytope) 
is precisely the convex hull of the 
permutation matrices.
We shall  denote by $DS(n)$ the set of all doubly stochastic $n\times n$ 
matrices and by $DS^\circ(n)$ the subset of matrices in $DS(n)$ 
for which all entries lie in the open interval $(0,1)$.
In what follows given $n$ numbers $s_1,\dots,s_n$ we let $\vec s $ be the column vector with entries $s_i$, and $\vec e_m$ be the $m$th coordinate vector.

The following interpolation theorem plays a crucial role in the proof of Theorems \ref{powerthm} and \ref{polthm}.
\begin{theorem}\label{interpolationtheorem}
Suppose we are given $m\in \{1,\dots, n\}$ and 
$\delta_1,\dots, \delta_n \in \bbR$ so that
the numbers $\delta_i$ with  $i\neq m$ are not all equal.
Let $0<r\le 1$, and 
let $q_1,\dots, q_n\in [r,\infty]$ such that  $\sum_{i=1}^n q_i^{-1}= r^{-1}$. Let $V$ be an $r$-convex Lorentz space, and let 
$\overline X=(X_0, X_1)$ be a couple of compatible complete quasi-normed spaces.
Let $T$ be a multilinear operator defined on $n$-tuples of $X_0+X_1$ valued sequences
and suppose that  for every permutation $\vpi$ on $n$ letters we have the
inequality
\begin{equation}\label{Thypothesis}
\|T(f_{\vpi(1)},\dots, f_{\vpi(n)})\|_{V} \le 
 \|f_m\|_{\ell_{\delta_m}^r(X_1)}
\prod_{i\neq m} \|f_i\|_{\ell_{\delta_i}^r(X_0)}\,.
\end{equation}
Then  for every $A\in DS^\circ(n)$ and every $B\in DS(n)$ such that  
$$B\vec e_m= (r/q_1,\dots, r/q_d)^T$$
 there is  $C=C(A,B,\vec\delta,r)$ so that
for  $\vec s=BA \vec\delta$ and $\vec \theta= BA \vec e_m$  
\begin{equation}\label{Tconcl}
\|T(f_1,\dots, f_{n})\|_{V} \le C
\prod_{i=1}^n \|f_i\|_{\ell_{s_i}^{q_i}(\overline X_{\theta_i,q_i})}\,,
\end{equation}
for all $(f_1,\dots, f_n)\in \prod_{i=1}^n\ell_{s_i}^{q_i}(\overline X_{\theta_i,q_i})$.

In particular 
\begin{equation}\label{Tconclbalanced}
\|T(f_1,\dots, f_{n})\|_{V} \lc
\prod_{i=1}^n \|f_i\|_{\ell_{\sigma}^{nr}(\overline X_{\frac 1n, nr})}\,,\quad
\sigma = \frac{1}{n}\sum_{i=1}^n\delta_i\,.
\end{equation}
\end{theorem}

Here, and in what follows we write $\lc$ if the inequality involves an implicit constant. 
For the proof of our restriction estimates only the special case 
\eqref{Tconclbalanced} is used; it follows from \eqref{Tconcl} by choosing $a_{ij}=b_{ij}=1/n$ for all $i,j$.

\subsubsection*{Relevance for the adjoint restriction operator}
One would like to extend the  proof of the endpoint estimate for the adjoint restriction operator 
 in \cite{bos1} by using weighted Lorentz spaces, but there is the  immediate difficulty
  that the real interpolation spaces of weighted Lebesgue or 
Lorentz spaces may not be  weighted Lorentz spaces, and other scales of spaces 
have to be considered (cf. the papers by Freitag \cite{freitag} and Lizorkin 
  \cite{liz} on interpolation spaces of weighted $L^p$ spaces).

Let $X$  be a Lorentz space of functions on an interval $I$ (with  Lebesgue measure), and a positive measurable weight function $w$ on $I$. Let $\Omega[w,k]=\{t\in I: 2^k\le w(t)<2^{k+1}\}$. We define the block Lorentz 
space 
 $b^q_s(w,X)$ to be 
the space of measurable functions for which 
\Be\label{blockLorentzdef}
 \|f\|_{b^q_s(w,X)}:= \Big(\sum_{k\in \bbZ}
\big[ 2^{ks} \|\chi_{\Omega[w,k]} f\|_X\big]^q\Big)^{1/q}
\Ee
is finite.  These  spaces arise in  real 
interpolation of weighted Lorentz spaces with change of measure (see
\cite{akmnp}, \cite{akn}). 
We are not necessarily interested in the block Lorentz spaces per se,
but use them as a vehicle to prove our  result on $L^p(w)= b^p_{1/p}(w,L^p)$.

The connection with results on 
$\ell^q_s(X)$ spaces 
  is immediate, namely 
 $b^q_s(w,X)$ is a retract of $\ell^q_s(X)$:
Define 
$\imath:b^q_s(w,X)
\to \ell^q_s(X)$ 
by $[\imath(f)]_k=\chi_{\Omega[w,k]} f$ and $\vars:\ell^q_s(X)\to b^q_s(w,X)$ by  
$\vars (F)= \sum_k \chi_{\Omega[w,k]} F_k$ then $\imath$ and $\vars$ have operator norm $1$ 
 and $\vars\circ\imath$ is the identity operator on $b^q_s(w,X)$;
 moreover $[\imath\circ \vars (F)]_k = \chi_{\Omega[w,k]}F_k$.
If $L$ is a linear operator mapping $b^q_s(w,X)$ 
boundedly 
to a quasi-normed space $V$ then $L\circ \vars: \ell^q_s(X)\to V$ 
and if $\cL$ 
is a linear operator mapping $\ell^q_s(X)$  to $V$ then 
$\cL\circ\imath: b^q_s(w,X)\to V$. 
Analogous observations can be made for multilinear operators acting on 
products
 of such spaces. 
Thus Theorem \ref{interpolationtheorem} implies  an immediate analog for multilinear operators acting on 
$\prod b^{q_i}_{s_i}(w,\overline X_{\vartheta_i,q_i})$
which will be used in our estimates for adjoint restriction operators.

\subsubsection*{This paper.} In \S\ref{optimality} we discuss the optimality of affine arclength measure in estimates for the Fourier restriction operators associated with curves.  In \S\ref{multlinsymsect} we prove 
Theorem
\ref{interpolationtheorem}.
In  \S\ref{hyp} we formulate geometrical 
 hypotheses for our  main  result on Fourier restriction from which
Theorems \ref{powerthm} and \ref{polthm} can be derived.
This result  is proved in \S \ref{proofofmaintheorem}. 
Theorem \ref{powerthm}  is proved in \S\ref{powertheoremsect}. 
In \S\ref{simpletypethms} 
we make some observations on curves of 
simple type and prove Theorem \ref{polthm}. In \S\ref{Drury-estimate}
we give the proof of  the  partial result for general polynomial curves alluded to above.
Some background needed for the interpolation section is provided in
Appendix \ref{appendixinterpol}.

\section{Optimality of the affine arclength measure}\label{optimality}
Let $\tau(t)$ be as in \eqref{weight}. 
For $p>1$ 
  let $\sig(p)=\frac{2p'}{d^2+d}$, with $p'=\frac{p}{p-1}$
(the critical $\sigma$ for $L^p\to L^\sigma$  boundedness of Fourier restriction 
with respect to Lebesgue measure in the nondegenerate case).
 In particular
$\sigma(p_d)=p_d$ for $p_d=\frac{d^2+d+2}{ d^2+d}$.  
\begin{proposition} \label{optprop}
Let $I$ be an interval and $\gamma:I\to \bbR^d$ be of class $C^{d}$.
Let $\mu$ be a positive Borel measure on $I$
and suppose that  the inequality
\Be\label{critpqineq} 
\Big(\int_I \big|\widehat f\circ\gamma\big|^{\sigma(p)}  \,d\mu \Big)^{1/\sigma(p)}
\le B \|f\|_{L^{p,1}}
\Ee holds for all $f\in L^{p,1}(\bbR^d)$.

Then $\mu$ is absolutely continuous with respect to Lebesgue measure 
on $I$, so that $d\mu =\omega(t) dt$ for a nonnegative  locally integrable
 $\omega$, and   
there exists 
 a constant $C_d$   so that
\Be\label{ineqforomega}
\omega(t)\le C_d B^{\sigma(p)} |\tau(t)|^{\frac {2}{d^2+d}}
 \Ee
for almost every $t\in I$.
\end{proposition}
\begin{proof} 
We argue as in the proof of  Proposition 2 in 
\cite{obcurv} and use a \lq Knapp  example\rq \ to see that 
\eqref{critpqineq}  implies
\begin{equation}\label{e2} 
\int \chi_P(\gamma(t)) \, d\mu(t)  \le C_1(d) B^{\sigma(p)}  |P|^{\frac{2}{d^2 +d}} 
\end{equation}
for any  parallelepiped $P$. Indeed if $P=AQ+b$ where $Q=[0,1]^d$, 
 $b\in \bbR^d$ and $A$ is an invertible linear transformation then we 
 choose $f$ so that $\widehat f(\xi)= \exp(-|A^{-1}(\xi-b)|^2)$. Now 
$|\widehat f(\xi)|\ge  e^{-d} $ for $\xi\in P$, and  
$\|f\|_{L^{p,1}} \le C_2(d) |\det (A)|^{1/p'}$, 
 and then \eqref{e2} is an immediate 
consequence of  the relation $\sigma(p)/p'= 2/(d^2+d)$ and $|P|=|\det A|$.

We first show that $\mu$ is absolutely continuous with respect to Lebesgue measure. Let $I'$ be a compact subinterval of $I$. 
Absolute continuity follows if we can show that
$\mu(J)\le C(I')|J|$ for every subinterval $J$ of $I'$ with length $|J|<1/2$.
By the Radon-Nikodym theorem 
$d\mu=\omega(t) dt $ with locally integrable  $\omega$
(in fact $\omega$ will be  locally bounded by the estimate on $\mu(J)$).

 Fix such a $J\subset I'$ and let $t$ be the center of $J$, and let $|J|=2h$.
Consider the Taylor expansion
\begin{equation}\label{e3}
\gamma (t+u )=\sum_{j=0}^d \frac{u^j}{j!}\,\gamma^{(j)}(t)+o(u^d)\, .
\end{equation}
Let $K=K_t$ denote the  dimension of the linear span $V_t$  of 
$\gamma'(t),...,\gamma^{(d)}(t)$.
Choose $1\le j_1<\dots<j_K\le d$ so that the span of 
$\gamma^{(j_1)}(t), \dots, \gamma^{(j_K)}(t)$ is equal to $V_t$ and so that
for each $j=1,\dots d$ the vector 
$\gamma^{(j)}(t)$ belongs to $\text{span}(\{\gamma^{(j_k)}(t),  j_k \le j\})$.
Choose an orthonormal basis  $\{v_k(t) \}_{1\le k\le d}$  so that
$\text{span} (\{v_1(t),\dots, v_l(t)\})$ is equal to 
$\text{span}(\{\gamma^{(j_k)}(t),  k=1,\dots, l\})$, for $l=1,\dots, K$.

Then there is a constant $C$ (depending on $I'$ and the $C^d$ bounds of $\gamma$) so that $\gamma(s)$, $s\in J$,  belongs to the 
parallelepiped 
$$\cP_C(h,t)=\gamma(t)+\big\{ \sum_{j=1}^d C b_j v_j(t): 0\le b_j\le h^j\}$$
which has volume $O(h^{\frac{d^2+d}{2}})$. By 
\eqref{e2} we get $$\mu(J) \le C_{1}(d) |\cP_C(h,t)|^{\frac{2}{d^2+d}}\le 
C(d,I',\gamma) |h|$$  which shows the absolute continuity of $\mu$.

In order to obtain  
\eqref{ineqforomega} it suffices, by the Lebesgue differentiation theorem,
 to prove
\Be \label{limsup}
\limsup_{h\to 0+}
\frac{1}{h}\int_{0}^{h}\omega(t+u) du 
 \le C_d B^{\sigma(p)}|\tau(t)|^{\frac{2}{d^2+d}}\,
\Ee
for every 
$t$ in the interior of $I$. 
In what follows fix such a $t$ and consider the Taylor expansion \eqref{e3}.
We distinguish the cases $\tau(t)=0$ and $\tau(t)\neq 0$.

If  $\tau (t)=0$ then $K_t\le d-1$ and using the orthonormal basis above the Taylor expansion can be rewritten as
$$
\gamma (t+u )=\gamma(t)+ \sum_{l=1}^K (c_l(t) u^{j_l}+g_l(u,t)u^d)\,v_l(t)
+ 
u^d \sum_{l=K+1}^{d} g_l(u,t)  v_l(t)
$$
where 
$g_l(u,t)\to 0$ as $u\to 0$.
Let  $\rho(h,t)= \max_{K+1\le l\le d} \sup_{0\le u\le h}|g_l(u,t)|$ and 
\begin{multline*}
P(h,t,C)=\gamma (t)+\big\{\sum_{k=1}^d Cb_k v_k(t): \\
0\le b_k \le h^{j_k}, k=1,\dots, K;  |b_k |\le \rho(h,t) h^d, k=K+1,\dots,d \big\}.
\end{multline*}
If $C$ is sufficiently large then there is $h_0(t)>0$ so that 
$\gamma (t+u )\in P(h,t,C)$ whenever  
$h\le h_0(t)$ and  $0\le u\le h$. Also 
$|{P(h,t,C)}|\lc h^{(d^2 +d)/2}\rho(h,t)$. Thus 
\eqref{e2} yields 
 $$\int_0^{h}w(t+u )\, du \lc B^{\sigma(p)}
h \rho(h,t)^{\frac{2}{d^2+d}}=o(h)$$ and 
we have verified \eqref{limsup} for the case   $\tau(t)=0$.

If $\tau(t)\neq  0$ we  may replace the above orthonormal basis by the basis 
$\gamma'(t)$,...,$\gamma^{(d)}(t)$ to rewrite  the Taylor expansion \eqref{e3}  as
$$
\gamma (t+u )=\sum_{j=0}^d \frac{u^j + u^d e_j(u,t)}{j!}\,\gamma^{(j)}(t)\, 
$$
where $\lim_{u\to 0+} |e_j(u,t)|=0$.
Let
$$P(h,t):=\gamma (t)+
\big\{\sum_{j=1}^d \frac{b_j}{j!}\,\gamma^{(j)}(t): 
0\le b_j \le 2h^j \big\}\, .
$$
Then $|P(h,t)|= C_3(d) h^{\frac{d^2+d}{2}}|\tau(t)|$ with $C_3(d)=2^d\prod_{j=1}^d\frac{1}{j!}$ and there 
is $h_0(t)>0$  so that for $ h<h_0(t)$ we have 
$\gamma(t+u)\in P(h,t)$ for $0\le u\le h$.
Thus, by \eqref{e2}  we see that for $h\le h_0(t)$
$$\int_{0}^{h}\omega(t+u) du \,\le \,C_1(d) B^{\sigma(p)}  h \big(C_3(d)|\tau(t)|
\big)^{\frac{2}{d^2 +d}}$$
which yields \eqref{limsup} in the case $\tau(t)\neq 0$.
\end{proof}

\section{Interpolation of multilinear operators with symmetries}
\label{multlinsymsect}
We shall now prove several  lemmata involving real interpolation
of multilinear operators with symmetry that have values in an $r$-convex quasi-normed space $V$. These will  lead  to the proof of Theorem
\ref{interpolationtheorem}. The reader may  consult  Appendix \ref{appendixinterpol} for some 
results from interpolation theory needed here.

The following notation, for a couple $\overline X=(X_0, X_1)$ of compatible quasinormed spaces, will be convenient. Set, for $0<q\le\infty$, 
\begin{equation}\label{widetildeconvention}
\widetilde X_{\theta,q}=\begin{cases} X_0, &\text{ if } \theta=0,
\\
\overline X_{\theta,q},
&\text{ if } 0<\theta<1,
\\
X_1, &\text{ if } \theta=1.
\end{cases}
\end{equation}
With this notation we formulate a version of Lemma \ref{thetar} for operators with symmetry.  

\begin{lemma}\label{shuffling}
Suppose $r\le 1$, and $\delta_1,\dots, \delta_n\in \bbR$. Let $(X_0,X_1)$ be a couple of compatible complete quasi-normed spaces.
Let $T$ be a multilinear operator defined on $n$-tuples of $X_0+X_1$ valued sequences, with values in an $r$-convex  space $V$
and suppose that  for every permutation $\vpi$ on $n$ letters we have the
inequality
\begin{equation}\label{endpthypothesis}
\|T(f_{\vpi(1)},\dots, f_{\vpi(n)})\|_{V} \le 
 \|f_1\|_{\ell_{\delta_1}^r(X_1)}
\prod_{i=2}^n \|f_i\|_{\ell_{\delta_i}^r(X_0)}\,.
\end{equation}

Then  there is a constant $C$ such  that 
for every doubly stochastic matrix $A=(a_{ij})_{i,j=1,\dots n}$,
for $s_i=\sum_{j=1}^n a_{ij}\delta_j$, $\theta_i=a_{i,1}$, $i=1,\dots,n$  and every permutation
$\vpi$, 
\begin{equation}\label{shufflingconcl}
\|T(f_{\vpi(1)},\dots, f_{\vpi(n)})\|_{V} \le C
\prod_{i=1}^n \|f_i\|_{\ell_{s_i}^r(\widetilde X_{\theta_i,r})}\,.
\end{equation}
\end{lemma}

\begin{proof}
Because of the permutation invariance of the assumption it suffices to prove
\eqref{endpthypothesis}  for $\vpi=\text{id}$.

The assumption says that 
$\|T[g_1,\dots, g_n]\|_V$
is dominated by $\|g_{\vpi^{-1}(1)}\|_{\ell^r_{\delta_1}(X_1)}$ $\times$
$\prod_{k=2}^n\|g_{\vpi^{-1}(k)}\|_{\ell^r_{\delta_k}(X_0)}$.
This can be rewritten as 
\Be\label{assumprewritten}
\begin{aligned} 
\|T[g_1,\dots, g_n]\|_V &\le 
\prod_{i=1}^n\|g_i\|_{\ell^r_{\delta_{\vpi(i)}}(\widetilde X_{\theta_i,r})}
\\ &\text{where $\theta_i=1$  if $\vpi(i)=1$ and $\theta_i=0$ if $\vpi(i)\neq 1$; }
\end{aligned}
\Ee
recall that by definition $\widetilde X_{0,r}=X_0$ and $\widetilde X_{1,r}=X_1$.
Let $P_\vpi$ be the permutation matrix which has
$1$  in the 
positions $(i,\vpi(i))$, $i=1,\dots, n$, and $0$ in the other positions. 
Then the conditions  $s_i= \delta_{\vpi(i)}$ and $\theta_i$ is as in \eqref{assumprewritten} can be rewritten as  $\vec s= P_\vpi \vec\delta$ and 
$\vec \theta= P_\vpi \vec e_1$ (here the vectors are all understood as columns).

For a doubly stochastic matrix $A$  let $\cH(A)$ be the statement
that the conclusion
$$
\|T(f_{1},\dots, f_{n})\|_{V} \le \cC(A)
\prod_{i=1}^n \|f_i\|_{\ell_{s_i}^r(\widetilde X_{\th_i,r})}
$$ holds for the vectors
$\vec s= A\vec\delta$, $\vec \th= A\vec  e_1$.
Now \eqref{assumprewritten}  is just saying 
that the statement $\cH(P_\vpi)$ holds.
By Birkhoff's theorem every $A\in DS(n)$ is a convex combination of permutation matrices and therefore  the general statement in
\eqref{shufflingconcl} follows immediately from repeated applications 
of a  {\it convexity property:}  Namely,
  if $\cH(A^{+})$ and $\cH(A^{-})$ hold for two doubly stochastic matrices $A^{+}$
 and $A^{-}$ then the statement $\cH((1-\gamma)A^{+}+\gamma A^{-})$ holds for $0< \gamma< 1$.

We now verify  this convexity property.
Let $A^{+}$ and $A^{-}$ be doubly stochastic matrices for which
$\cH(A^{+})$ and $\cH(A^{-})$ hold,
thus we have
\begin{align*}
\|S(f_{1},\dots, f_{n})\|_{V} &\le
\cC(A^{+})
\prod_{i=1}^n \|f_i\|_{\ell_{s_{i}^{+}}^r(\widetilde X_{\th_{i}^{+},r})}
\\
\|S(f_{1},\dots, f_{n})\|_{V} &\le
\cC(A^{-})
\prod_{i=1}^n \|f_i\|_{\ell_{s_{i}^{-}}^r(\widetilde X_{\th_{i}^{-},r})}
\end{align*}
for column vectors  $\vec s^{\pm}=A^\pm \vec\delta$, $\vec \th^{\pm}= 
A^{\pm} \vec e_1$.
By  taking generalized geometric means we also have
$$\|S(f_1,\dots, f_n)\|_{V}
\le \cC(A^{+})^{1-\gamma} \cC(A^{-})^{\gamma}
\prod_{i=1}^n \big(\|f_i\|_{\ell_{s_{i}^{+}}^r(\widetilde X_{\th_{i}^{+},r})}^{1-\gamma}
\|f_i\|_{\ell_{s_{i}^{-}}^r(\widetilde X_{\th_{i}^{-},r})}^\gamma\big)\,
$$
for $0<\gamma<1$.

Now let temporarily $W_{i,0}= \ell_{s_{i}^{+}}^r(\widetilde X_{\th_{i}^{+},r})$,
and $W_{i,1}= \ell_{s_{i}^{-}}^r(\widetilde X_{\th_{i}^{-},r})$.
By the last displayed formula and  Lemma \ref{thetar} we get
$$\|S(f_1,\dots, f_n)\|_{V} \le C\,\cC(A^{+})^{1-\gamma} \cC(A^{-})^{\gamma}
\prod_{i=1}^n \|f_i\|_{(W_{i,0},W_{i,1})_{\gamma,r}}
\,.$$
By the reiteration theorem
we have
$$\big(\widetilde X_{\th_{i}^{+},r}, \widetilde X_{\th_{i}^{-},r}\big)_{\gamma,r}=
\widetilde X_{(1-\gamma)\th_i^{+}+\gamma\th_i^{-},r}, \quad 0<\gamma<1\,,
$$
and then, by Lemma \ref{quasi} there is the continuous embedding
$$
\ell^r_{(1-\gamma) s_i^{+}+\gamma s_i^{-}} (\widetilde X_{(1-\gamma)\th_i^{+}+
\gamma\th_i^{-},r})
\hookrightarrow  (W_{i,0},W_{i,1})_{\gamma,r}\,.
$$
Hence, for some $\cC$
\begin{equation}\label{interpolconclusion}
\|S(f_1,\dots, f_n)\|_{V} \le \cC
\prod_{i=1}^n \|f_i\|_{\ell^r_{(1-\gamma) s_i^{+}+\gamma s_i^{-}}
(\widetilde X_{(1-\gamma)\th_i^{+}+\gamma\th_i^{-},r}) }\,.
\end{equation}
Let $A^{(\gamma)}= (1-\gamma) A^{+}+\gamma A^{-}$ then
$(1-\gamma) s_i^{+}+\gamma s_i^{-}= \sum_{j=1}^n a_{ij}^{(\gamma)} \delta_j$
and
$(1-\gamma) \th_i^{+}+\gamma \th_i^{-}=  a_{i1}^{(\gamma)}$ and thus
\eqref{interpolconclusion} is just 
$\cH((1-\gamma) A^{+}+\gamma A^{-})$.
\end{proof}

We shall now apply an iterated  version of the interpolation method by 
Christ \cite{Ch}  to upgrade $n-1$ of the $n$ spaces
$\ell_{s_{i}}^r(\widetilde X_{\th_{i},r}) $ to
$\ell_{s_{i}}^\infty(\widetilde X_{\th_{i},\infty}) $, provided that the parameters
correspond to doubly stochastic matrices in the
{\it interior} of the Birkhoff polytope,
the set 
 $DS^\circ(n)$ 
 of doubly stochastic $n\times n$  matrices
$A=(a_{ij})$ for which all entries lie in the open interval $(0,1)$.
In the following lemma we  assume the conclusion of the previous lemma and also an additional assumption on $\vec\delta$.

\begin{lemma}\label{upgradelemma}
Suppose $n\ge 3$, $0<r\le 1$, and $\delta_1,\dots, \delta_n\in \bbR$.
Assume that there are two indices $i_1,i_2$ with $2\le i_1<i_2\le n$ so that
$\delta_{i_1}\neq\delta_{i_2}$. Let $X_0$, $X_1$ be compatible, 
complete quasi-normed spaces and let $T$ be an $n$-linear operator defined on 
$n$-tuples of
$X_0+X_1$-valued sequences,  with values  in an $r$-convex space $V$.
Suppose that for every $A\in DS^\circ(n)$ there is $C(A)$ such that
\begin{equation}\label{shufflingassump}
\|T(f_1,\dots, f_n))\|_{V} \le C (A)
\prod_{i=1}^n \|f_i\|_{\ell_{\sigma_i}^r(X_{\mu_i,r})}
\end{equation}
whenever  $\vec \sigma=A\vec\delta$, $\vec\mu=A \vec e_1$.

Then  for every $A\in DS^\circ(n)$ there is $\widetilde C(A)$ such  that
\begin{equation}\label{interpolupgrade}
\|T(f_1,\dots,f_n)\|_{V} \le \widetilde C(A)
\|f_1\|_{\ell_{s_1}^r(\overline X_{\th_1,r})}
\prod_{i=2}^n \|f_i\|_{\ell_{s_i}^\infty(\overline X_{\th_i,\infty})}
\end{equation}
with $\vec s=A\vec\delta$, $\vec\theta=A\vec e_1$.
\end{lemma}

\begin{proof} Let $A\in DS^\circ(n)$, and let
$2\le k\le n+1$.
Let $\cH_{n+1}(A)$ denote the statement that
\eqref{shufflingassump} is true for
$\vec \sigma=A\vec\delta$, $\vec\mu=A \vec e_1$.

Let $\cH_{n+1}$ denote the hypothesis \eqref{shufflingassump} 
for all $A\in DS^\circ(n)$. 
For $2\le k\le n$ let $\cH_k(A)$
denote the statement that there is  $C$ depending on $A$ so that  the inequality
\begin{equation}\label{interpolk}
\|T(f_{1},\dots, f_{n})\|_{V} \le C
\Big(\prod_{i=1}^{k-1}\|f_j\|_{\ell_{s_i}^r(\overline X_{\th_i,r})}\Big)
\Big(\prod_{i=k}^n \|f_i\|_{\ell_{s_i}^\infty(\overline X_{\th_i,\infty})}\Big)
\end{equation}
holds for all
$(f_1,\dots, f_{k-1})
\in \prod_{i=1}^{k-1}\ell_{s_i}^{r}(\overline X_{\theta_i,r})$,
$(f_k,\dots, f_n)\in  \prod_{i=k}^{n}\ell_{s_i}^{\infty}(\overline X_{\theta_i,\infty})
$, under the condition  
$\vec s=A\vec\delta$, $\vec\theta=A \vec e_1$.
Let $\cH_k$ denote the statement that  $\cH_k(A)$ holds for all
$A\in DS^\circ(n)$.

We seek to prove $\cH_2$.
In what follows we thus need to  show for $2\le k\le n$, $A\in DS^\circ(n)$  
that
 $\cH_{k+1}$
implies $\cH_k(A)$.
We assume in our writeup that $k\le n-1$
but the proof carries through to cover the initial  case  $k=n$
if we interpret $\prod_{i=n+1}^n \dots$ as $1$.

Assuming $\cH_{k+1}$ 
we shall first prove  a preliminary inequality $\cH^{\prel}_{k}(A)$, namely
\begin{multline}\label{interpolkintermed}
\|T(f_{1},\dots, f_{n})\|_{V} \le \\C
\Big(\prod_{i=1}^{k-1}\|f_j\|_{\ell_{s_i}^r(\overline X_{\th_i,r})}\Big)
\|f_k\|_{\ell_{s_k}^r(\widetilde X_{\th_k,\infty})}
\Big(\prod_{i=k+1}^n \|f_i\|_{\ell_{s_i}^\infty(\overline X_{\th_i,\infty})}\Big).
\end{multline}
We denote by
$\cH^{\prel}_{k}$ the statement that $\cH^{\prel}_{k}(A)$ holds for every
$A\in DS^\circ(n)$.

\subsubsection*{Proof that $\cH_{k+1}$  implies $\cH_k^\prel$}
Fix $A\in DS^\circ(n)$ and let  $\eps>0$ with the property that all entries of $A$ lie in the open interval $(2\eps, 1-2\eps)$. We define two $n\times n$ matrices $A^+=A^+(k)$ and $A^-=A^{-}(k)$   by letting
\begin{equation}
a^\pm_{\mu\nu}=\begin{cases}
a_{\mu\nu} &\text{ if } (\mu,\nu) \notin \big\{(1,1), (1,k), (k,1), (k,k)\big\}
\\
a_{\mu\nu}\pm \eps &\text{ if } (\mu,\nu)=(1,1) \quad\text{or} \quad
 (\mu,\nu)=(k,k)\,,
\\
a_{\mu\nu}\mp \eps &\text{ if } (\mu,\nu)=(1,k) \quad\text{or} \quad
 (\mu,\nu)=(k,1)\,.
\end{cases}
\end{equation}
It is easy to see that $A^\pm$ belong to $DS^\circ(n)$.
Also if $\vec s=A\vec\delta$, $\vec\theta=A\vec e_1$,
and
$\vec s^\pm=A^\pm\vec\delta$, $\vec\theta^\pm=A^\pm\vec e_1$,
then
$s^\pm_i=s_i$ if $i\notin \{1,k\}$,
$s^\pm_1=s_1\pm\eps(\delta_1-\delta_k)$,
$s^\pm_k=s_k\pm\eps(\delta_k-\delta_1)$,
moreover $\th_i=\theta^\pm_i$ if
$i\notin \{1,k\}$,
$\th^\pm_1=\th_1\pm\eps$, $\th^\pm_k=\th_k\mp\eps$.

We interpolate the linear operator $L_k$ given by
\Be\label{Lk}g\mapsto L_kg=T(f_1,\dots, f_{k-1}, g, f_{k+1},\dots, f_n)\Ee
using the real interpolation $K_{\vth,q}$  method with parameters $\vth=1/2$ and $q=\infty$. Since  $\th_k^+\neq \th_k^{-}$  we have
by Lemma \ref{quasi} and the reiteration theorem
$$
\ell^r_{s_k}(\overline X_{\th_k,\infty})=
\ell^r_{s_k}\big((\overline  X_{\th_k^+,r}, \overline X_{\th_k^-,r})_{\frac{1}2,\infty}\big) 
\hookrightarrow 
\big( \ell^r_{s_k^+}(\overline X_{\th_k^+,r}),
\ell^r_{s_k^-}(\overline X_{\th_k^-,r})\big)_{\frac 12,\infty}\,.
$$
Thus we obtain by interpolation of $L_k$ (using $\cH_{k+1}(A^\pm)$)
\begin{multline*}
\|T(f_{1},\dots, f_n)\|_{V} \le C
\prod_{i=1}^{k-1}\Big(\|f_j\|_{\ell_{s_i^+}^r(\overline X_{\th_i^+,r})}^{1/2}
\|f_i\|_{\ell_{s_i^-}^r(\overline X_{\th_i,r})}^{1/2}\Big) \times\\
\|f_k\|_{\ell_{s_k}^r(\overline X_{\th_k,\infty})}
\Big(\prod_{i=k+1}^n \|f_i\|_{\ell_{s_i}^\infty(\overline  X_{\th_i,\infty})}\Big).
\end{multline*}
By Lemma \ref{thetar} and the reiteration theorem  
we get
\begin{multline*}
\|T(f_{1},\dots, f_n)\|_{V} \le\\ C'
\Big(\prod_{i=1}^{k-1}\|f_i\|_{\ell_{s_i}^r(\overline X_{\th_i,r})}\Big)
\|f_k\|_{\ell_{s_k}^r(\widetilde X_{\th_k,\infty})}
\Big(\prod_{i=k+1}^n \|f_i\|_{\ell_{s_i}^\infty(\overline X_{\th_i,\infty})}\Big).
\end{multline*}
This finishes the proof of the implication $\cH_{k+1}\implies \cH_k^\prel$.

\subsubsection*{Proof that $\cH_k^\prel$ implies $\cH_k$}
Fix $A\in DS^\circ(n)$ and let $\eps>0$ so that
all entries of $A$ lie in the open interval $(2\eps, 1-2\eps)$. We define two $n\times n$ matrices $A^+$ and $A^-$ (depending on $k$ and different from the ones in the first step)  by letting
\begin{equation}\label{secondmatrix}
a^\pm_{\mu\nu}=\begin{cases}
a_{\mu\nu} &\text{ if } (\mu,\nu) \notin \big\{
(1,i_1), (1,i_2), (k,i_1), (k,i_2)\big\}
\\
a_{\mu\nu}\pm \eps &\text{ if } (\mu,\nu)=(1,i_1) \quad\text{or} \quad
 (\mu,\nu)=(k,i_2)\,,
\\
a_{\mu\nu}\mp \eps &\text{ if } (\mu,\nu)=(k,i_1) \quad\text{or} \quad
 (\mu,\nu)=(1,i_2)\,.
\end{cases}
\end{equation}

Then $A^+$ and $A^-$ are in $DS^\circ(n)$. It is important for our argument that
the first column of $A^\pm$ is equal to the first column of $A$.
Let $\vec s^\pm =A^\pm\vec\delta$,
$\vec s =A\vec\delta$;
then $s_i^\pm =s_i$ for $i\notin \{1,k\}$ and
$s_1^\pm=s_1\pm(\delta_{i_1}-\delta_{i_2})$,
$s_k^\pm=s_k\pm(\delta_{i_2}-\delta_{i_1})$,
so that by the  assumption $\delta_{i_1}\neq \delta_{i_2}$ we have
  $s_k^+\neq s_k^-$
and $s_k$ is the
arithmetic mean of $s_k^+$ and $s_k^{-}$.
 Moreover $s_i^\pm=s_i$ if $i\notin\{1,k\}$.

We interpolate the linear operator $L_k$ as in \eqref{Lk}.
This time we use $\cH_k^\prel (A^\pm)$ and the formula
\Be\label{secondinterpol}
\big(\ell^r_{s_k^+}(\overline X_{\th_k,\infty}),
\ell^r_{s_k^-}(\overline X_{\th_k,\infty})\big)_{\frac 12,\infty}=
\ell^\infty_{s_k}(\overline X_{\th_k,\infty})
\Ee
which is a special case of formula
\eqref{seqinterpolfixedA} in the appendix.
This  yields
\begin{multline*}
\|T(f_{1},\dots, f_n)\|_{V} \le \\ C\Big(
\prod_{i=1}^{k-1}(\|f_i\|_{\ell_{s_i^+}^r(\overline  X_{\th_i,r})}^{1/2}
\|f_i\|_{\ell_{s_i^-}^r(\overline  X_{\th_i,r})}^{1/2})\Big)
\Big(\prod_{i=k}^n \|f_i\|_{\ell_{s_i}^\infty(\overline  X_{\th_i,\infty})}\Big).
\end{multline*}
By Lemma \ref{thetar} and the reiteration theorem   we also get
$$
\|T(f_{1},\dots, f_n)\|_{V} \le C'
\Big(\prod_{i=1}^{k-1}(\|f_i\|_{\ell_{s_i}^r(\overline  X_{\th_i,r})}\Big)
\Big(\prod_{i=k}^n \|f_i\|_{\ell_{s_i}^\infty(\overline  X_{\th_i,\infty})}\Big)
$$
which is $\cH_k(A)$.
\end{proof}



We are now in the position to give the
\begin{proof}[{\bf Proof of Theorem  \ref{interpolationtheorem}}]
We first reduce to the case  $m=1$. If $\vpi$ denotes the permutation that interchanges  $1$ and $m$, with $\vpi(i)=i$ for $i\notin\{ 1,m\}$, we define
\Be \label{Tvpidef}
T_\vpi[f_1,\dots, f_n]= T[f_{\vpi(1)},\dots, f_{\vpi(n)}].\Ee
If $T$ satisfies the assumptions of 
Theorem  \ref{interpolationtheorem}  then $T_\vpi$ satisfies 
the assumptions with the choice $m=1$ and the parameters 
$(\delta_1,\dots,\delta_n)$  replaced with 
$(\delta_{\vpi(1)},\dots,\delta_{\vpi(n)})$. By symmetry we get the statement for $T$ from the statement for $T_\vpi$.

After this reduction we may assume  $m=1$
in what follows.
Let $A\in DS^\circ(n)$.
For any $B\in DS(n)$ let $\cH_A(B)$ denote the statement 
that there is $ C>0$ so that the inequality 
\begin{equation*}
\|T(f_1,\dots, f_{n})\|_{V} \le C
\prod_{i=1}^n \|f_i\|_{\ell_{s_i}^{q_i}(\overline X_{\theta_i,q_i})}
\end{equation*}
holds for all $(f_1,\dots, f_n)\in \prod_{i=1}^n\ell_{s_i}^{q_i}(\overline X_{\theta_i,q_i})$, under the condition that
$$
\vec s= BA \vec\delta, \quad \vec \theta= BA \vec  e_1, \quad
\begin{pmatrix} r/q_1\\ \vdots\\r/q_n \end{pmatrix} =  B\vec e_1\,.
$$
As an immediate consequence of Lemma  \ref{transferlemma} we see that given $A\in DS^\circ(n)$ the matrices $B$ satisfying $\cH_A(B)$ satisfy a convexity
 property, namely,
if $\cH_A(B^{(0)})$, $\cH_A(B^{(1)})$ hold for $B^{(0)}\in DS(n)$,
$B^{(1)}\in DS(n)$ then for $0<\vth<1$,  
$B^{(\vth)}= (1-\vth) B^{(0)}+\vth
 B^{(1)}$ 
the statement  $\cH_A(B^{(\vth)})$ also holds. 
By Birkhoff's theorem  $\cH_A(B)$ holds for all $B\in DS(n)$ once we have shown it for all permutation matrices.

For this we first   apply 
Lemma \ref{shuffling}, and then, 
 by the conclusion of that lemma  we can apply  Lemma
\ref{upgradelemma} to the 
multilinear operators 
$T_\vpi$ in \eqref{Tvpidef}.
As a consequence there is, for every $A\in DS^\circ(n)$, 
 a constant $C(A)$ so that
for $s_j=\sum_{j=1}^n a_{ij}\delta_j$, $\th_i=a_{i,1}$ and every permutation
$\vpi$
\begin{equation}\label{shufflingconcl-2}
\|T(f_{\vpi(1)},\dots, f_{\vpi(n)})\|_{V} \le C
\|f_1\|_{\ell_{s_1}^r(\widetilde X_{\th_1,r})}\prod_{i=2}^n \|f_i\|_{\ell_{s_i}^\infty(\widetilde X_{\th_i,\infty})}\,.
\end{equation}
It is straightforward to check that this conclusion is exactly
statement $\cH_A(P)$ for all permutation matrices $P$. 
\end{proof}

\section{Hypotheses for the restriction theorem}
\label{hyp}

Given a parameter  interval $I$ we shall  consider a class $\fC$ of vector valued functions $\gamma: J_\gamma\to \bbR^d$ of class $C^d$
defined on subintervals $J_\gamma$ of $I$. For every $\gamma \in \fC$ the restrictions of $\gamma$ to subintervals of $J_\gamma $ will also be in  $\fC$.

\begin{definition}
(i) Let $J$ be an interval and let
$\kappa =(\kappa_1,\dots,\kappa_d)$ so that $\kappa_1\le \kappa_2\le\dots\le \kappa_d$ ,  $\kappa_d-\kappa_1\le |J|$ and one of the coordinates $\ka_i$ 
is equal to $0$. Let 
 $J^\kappa=\{t: t+\kappa_1\in J, \,t+\kappa_d\in J\}$.
Then given a curve $t\mapsto \gamma(t)\in \bbR^d$,
$t\in J$ we define the $\kappa$-offspring curve $\gamma_\ka$ on $J^\ka$ by
$$\gamma_\ka(t)= \sum_{j=1}^d\gamma(t+\kappa_j).$$

(ii) Let $\tau=\tau_\gamma$ be as in \eqref{weight}. We denote by 
$\tau_{\ga_\ka}$ the corresponding expression for the offspring curve, i.e.
$$\tau_{\gamma_\kappa}= \det\big(\sum_{j=1}^d \gamma'(t+\ka_j),\dots,
\sum_{j=1}^d \gamma^{(d)}(t+\ka_j)\big)\,.$$
\end{definition}

Let \Be\label{defofQ} Q=\frac{d^2+d+2}{2}\,\Ee
so  that $Q=p_d'$ with $p_d$ as in \eqref{gammaendpt}.
For $\gamma \in \fC$ defined on $I$ we shall consider  the adjoint operator 
$$\cE_w f(x) = \int_I e^{-\ic \inn{x }{ \gamma (t)}} f(t) w(t) dt $$
with $w(t)\equiv w_\gamma(t)=|\tau_\gamma(t)|^{2/(d^2+d)}$
and then examine the $L^Q(w;I)\to L^{Q,\infty}(\bbR^d)$ operator norms
(here we work with a fixed equivalent norm on $L^{Q,\infty}(\bbR^d)$).
We consider $L^{Q,\infty}(\bbR^d)$ as a normed space, the norm being
$\|h\|^{**}_{L^{Q,\infty}}=\sup_{t>0} t^{1/Q} h^{**}(t)$ where $h^{**}$ is as in
\eqref{maxrearr} with $\rho=1$. 
We also continue to use 
$\|h\|_{L^{Q,\infty}}=\sup \alpha(\meas (\{|h|>\alpha\}))^{1/Q}$,
the usual equivalent quasinorm (and the  constants in this equivalence 
 are  independent of the measure space).

We shall  make the a priori assumption that
\Be\label{defofB}
\cB\equiv \cB(\fC):=
\sup_{\gamma\in \fC}\sup_{\|f\|_{L^{Q}(w_\gamma)}\le 1}
\|\cE_w f\|_{L^{Q,\infty}}^{**}\Ee
is finite and the main goal is to give a geometric bound for the constant $\cB$.
We remark that the finiteness of $\cB$ has been shown in \cite{bos1} for certain classes of smooth curves with nonvanishing torsion. Once a  more effective bound for $\cB$ is  established one  can prove Theorems \ref{powerthm} and \ref{polthm} by limiting arguments.

We now formulate the   hypotheses of our main estimate.
Two  of them were relevant already in \cite{DM1}, \cite{DM2}.

\begin{hyps}\label{hypess}
Let $\fC$ be a class of curves with base interval $I_\circ$.
For $\gamma\in \fC$ defined on $I\subset I_\circ$
let \Be \label{Egamma}
E =\{(t_1,\dots, t_d):  t_1\in I, \,\,t_d\in I,  \,\, t_1<t_2<\dots<t_d \}.
\Ee

(i) There is $N_1\ge 1$ so that for every $\gamma \in \fC$ the map
$\Phi_\gamma:E\to \bbR^d$ with
\Be\label{multipl}
\Phi_\gamma(t_1,\dots, t_d)= \sum_{j=1}^d \gamma(t_j)
\Ee
is of multiplicity at most $N_1$.

(ii) Let $\cJ_{\Phi_\gamma}$ denote the Jacobian of  $\Phi_\gamma$,
$$\cJ_{\Phi_\gamma}(t_1,\dots, t_d)= \det \big(\gamma'(t_1),\dots, \gamma'(t_d)\big).
$$
Then there is $\fc_1>0$ 
such  that for every $(t_1,\dots, t_d)\in \cI^d$ with $t_1<\dots<t_d$ we have the inequality
\Be\label{firstmainhyp}
 |\cJ_{\Phi_\gamma}(t_1,\dots, t_d)| \ge \fc_1
\Big(\prod_{i=1}^d  \tau_\gamma(t_i)\Big)^{1/d}
\prod_{1\le j<k\le d} (t_k-t_j)\,.
\Ee

(iii) Every offspring curve of a curve in $\fC$ is (after possible reparametrization) the affine image of a curve in $\fC$.

(iv)
There is $\fc_2>0$ so that for every $\gamma\in \fC$ and every
offspring curve $\gamma_\ka$ of $\ga$ we have the inequality
\Be\label{torsionoffspringstr}
|\tau_{\gamma_\ka}(t)|\ge \fc_2 \max_{j=1,\dots, d}
|\tau_\gamma(t+\ka_j)|\,.
\Ee
\end{hyps}

Inequality \eqref{torsionoffspringstr}  is a strengthening of a weaker
 inequality which was used in
\cite{DM1}, \cite{DM2}, \cite{bos1}, \cite{bos2},  namely
\Be\label{torsionoffspring}
|\tau_{\gamma_\ka}(t)|\gc \prod_{j=1}^d |\tau_\gamma(t+\ka_j)|^{1/d}.
\Ee
The stronger inequality
allows us to replace the geometric mean on the right hand side of \eqref{torsionoffspring} by generalized geometric means
$\prod_{j=1}^d |\tau_\gamma(t+\ka_j)|^{\eta_j}$ for nonnegative $\eta_j$
with $\sum_{j=1}^d \eta_j=1$. Our main result is

\begin{theorem} \label{mainthm} Let $\fC$ be a class of curves satisfying Hypothesis \ref{hypess} and $\cB(\fC)<\infty$. Then 
\Be\label{cBbound} \cB(\fC)\le C(d, N_1,\fc_1^{-1}, \fc_2^{-1}).
\Ee
\end{theorem} 
For an explicit constant see \eqref{finalbound} below.

\section{Proof of Theorem \ref{mainthm}}\label{proofofmaintheorem}
Let $w\equiv w_\gamma$ define the affine arclength measure of $\gamma$.  We start with 
\begin{observation}\label{observationweight}
If  $\gamma\in \fC$ is defined on $I$  and  $\widetilde w$ is a nonnegative measurable weight  satisfying
$\widetilde w(t)\le w(t)$ then there is a constant $C$ such that
$$
\big\|\cE_{\widetilde w} f\big\|_{L^{Q,\infty}(\bbR^d)} \le C\cB \Big(\int_I|f(t)|^Q \widetilde w(t) \, dt\Big)^{1/Q}\,.
$$
\end{observation}

\begin{proof} One can use a  duality argument (as in the submitted version) 
 to deduce the  claim
from $\|\widehat g\circ \gamma\|_{L^{Q'}(\widetilde w)}\le 
\|\widehat g\circ \gamma\|_{L^{Q'}(w)}$.
The referee suggested 
a  simpler and more direct 
 argument: Write $\widetilde w=hw$ where $0\le h\le 1$ and use that 
$\cE_{\widetilde w} f=\cE_w(hf)$ and $h^Q\le h$.
\end{proof}


We aim to  prove estimates for the $d$-linear operator $\cM$ defined by
$$
\begin{aligned}
\cM[f_1,\dots, f_d](x)&= \prod_{i=1}^d \cE_w f_i(x)\\&=
\int_{I^d} \exp(-\ic \inn{x}{\sum_{j=1}^d\gamma(t_j)}) \prod_{i=1}^d[f_j(t_j)w(t_j)] \,dt_1\dots dt_d\end{aligned}
$$
Denote by $V(t_1,\dots, t_d)= \prod_{1\le i<j\le d}(t_j-t_i)$ the Vandermonde determinant 
and let,  for $l\in \bbZ$,  
$$E_l=\{(t_1,\dots, t_d)\in I^d: 2^{-l-1}<|V(t)|\le 2^{-l}\}.$$
Following the reasoning in \cite{baklee}, \cite{bos1} 
for the nondegenerate case we decompose  $\cM=\sum_{l\in \bbZ} \cM_l$ where 
\Be\label{Mldef}\cM_l[f_1,\dots, f_d](x)=
\int_{E_l} \exp(-\ic \inn{x}{\sum_{j=1}^d\gamma(t_j)})\prod_{i=1}^d[f_i(t_j)w(t_i)]\, dt_1\dots dt_d\,.
\Ee

An important ingredient
is an estimate for the sublevel sets of  the restriction of $V$ to $\bbR^{d-1}$, namely 
\Be\label{sublevelsetest}
\meas\big(\big\{(h_1,\dots, h_{d-1}):\,|h_1\cdots h_{d-1}|\prod_{1\le i<j\le d-1} 
|h_j-h_i|\le \alpha\big\}\big)\le C_d\alpha^{2/d}
\Ee
where the measure is Lebesgue measure in $\bbR^{d-1}$. This was proved in
\cite{DM1}  (\cf.  also the exposition in 
\cite{bos1}).

\begin{lemma}\label{Mlestimates}
(a) Let $\rho_i\in [2,\infty]$ be such that $\sum_{i=1}^d \rho_i^{-1}=1/2$.
Then
\Be \label{MlL2}
\|\cM_l [f_1,\dots, f_d]\|_{L^{2}}\lc (N_1/\fc_1)^{1/2} 2^{l\frac{d-2}{2d}} 
\prod_{i=1}^d\big\|f_i w^{\frac{3-d}{4}}\big\|_{\rho_i}.
\Ee

(b) Let $\eta_i\in [0,1]$ so that $\sum_{i=1}^d \eta_i=1$, and $q_i$ such that $\sum_{i=1}^d q_i^{-1}= 1/Q$.
Then 
\Be \label{MlLQinfty}
\|\cM_l [f_1,\dots, f_d]\|_{L^{Q,\infty}}\lc \cB \fc_2^{-1/Q} 2^{-2l/d}  
\prod_{i=1}^d\big\|f_i w^{1- \frac{\eta_i}{Q'}}\big\|_{q_i}.\Ee
\end{lemma}

\begin{proof}
For every permutation $\vpi$ on $d$  letters  set $$E^{\vpi}=
\{(t_1,\dots, t_d): t_{\vpi(1)}<\dots<t_{\vpi(d)}\}.$$
By assumption the  map $\Phi:
(t_1,\dots, t_d)\mapsto \sum_{i=1}^d\gamma(t_i)$ is of bounded multiplicity $d!N_1$ on 
$\cup_{\vpi} E^{\vpi}$. We apply the change of variable, followed by Plancherel's theorem, followed by the inverse change of variable to bound
$$\|\cM_l[f_1,\dots, f_d]\|_2\lc \Big(N_1\int_{E_l}
\big|   \prod_{i=2}^d[f_i(t_i)w(t_i)\big|^2 \frac{dt_1\dots dt_d}{|J_{\Phi}(t_1,\dots, t_d)|}
\Big)^{1/2}.$$
By our assumption \eqref{firstmainhyp}, the right hand side is dominated by
\begin{align*}
&\Big(N_1\int_{E_l}
\big|  \prod_{i=1}^d f_i(t_i)w(t_i)\big|^2 \frac{dt_1\dots dt_d}{\fc_1|V(t)|\prod_{i=1}^d w(t_i)^{\frac{d+1}{2}}}
\Big)^{1/2}
\\
&\lc \Big(\frac{N_1}{\fc_1}\int_{E_l}
\big|  \prod_{i=1}^d f_i(t_i)w(t_i)^{\frac{3-d}{4}}
\big|^2\frac{ dt_1\dots dt_d}{|V(t)|}
\Big)^{1/2}
\end{align*}
The sublevel set estimate \eqref {sublevelsetest} with $\alpha=2^{-l}$
yields 
$$\Big(\int_{E_l}
\big| \prod_{i=1}^dg_i(t_i)\big|^2 \frac{dt_1\dots dt_d}{|V(t)|}
\Big)^{1/2}\lc 2^{l\frac{d-2}{2d}}  \|g_d\|_2 \prod_{i=1}^{d-1} \|g_i\|_\infty\,.$$
By symmetry we get a similar statement with  the variables permuted and 
then by complex interpolation also 
\Be\label{relaxingLinfty}
\Big(\int_{E_l}
\big|   \prod_{i=1}^d g_i(t_i)\big|^2 \frac{dt_1\dots dt_d}{|V(t)|}
\Big)^{1/2}\lc 2^{l\frac{d-2}{2d}}   \prod_{i=1}^d\|g_i\|_{\rho_i}
\Ee
where $\sum_{i=1}^d \rho_i^{-1}=1/2$. If we apply this statement with 
$g_i= f_iw^{\frac{3-d}{4}}$ we obtain
\eqref{MlL2}.

To prove \eqref{MlLQinfty} 
it suffices to 
show  that for fixed $(\eta_1,\dots,\eta_d)$ with 
$\eta_i\ge 0$ and $\sum_{i=1}^d \eta_{i}=1$
\Be \label{MlLQinftyunbalanced}
\|\cM_l [f_1,\dots, f_d]\|_{L^{Q,\infty}}
\lc \cB \fc_2^{-1/Q} 2^{-2l/d}  
\big\|f_1 w^{1- \frac{\eta_1}{Q'}}\big\|_Q
\prod_{i=2}^{d}\big\|f_i w^{1- \frac{\eta_i}{Q'}}
\big\|_{\infty}.\Ee
Once this inequality is verified, we also  get, by the symmetry of $\cM_l$, the bound
$\cB \fc_2^{-1/Q} 2^{-2l/d}\|f_m w^{1- \frac{\eta_m}{Q'}}\|_Q
\prod_{i\neq m}\big\|f_i w^{1- \frac{\eta_i}{Q'}}\|_\infty$ and then the inequality 
\eqref{MlLQinfty}  follows by complex  interpolation.

Let, for any permutation $\vpi$ on $d$ letters
$$E^\vpi_l=
\{(t_1,\dots, t_d)\in I^d: t_{\vpi(1)}<\dots<t_{\vpi(d)},\,\, 2^{-l-1}<V(t_1,\dots, t_d)
\le 2^{-l}\}.$$
To prove
\eqref{MlLQinftyunbalanced}  
we split $\cM_l=\sum_\vpi M_l^\vpi$ where
$$M_l^\vpi [f_1,\dots, f_d](x)=
\int_{E^\vpi_l} \exp(-\ic \inn{x}{\sum_{j=1}^d\gamma(t_j)})
\prod_{i=1}^d[f_i(t_i)w(t_i)] \,dt_1\dots dt_d\,.
$$
We need to bound
$\|\cM_l^\vpi [f_1,\dots, f_d]\|_{L^{Q,\infty}}$ by the right hand side of 
\eqref{MlLQinftyunbalanced}.

Now let $\nu=\vpi^{-1}(1)$ and let $H_\nu$ be the hyperplane
$\{(\kappa_1,\dots,\kappa_d): \kappa_\nu=0\}$. 
We change variables
$$s=t_1,\quad \kappa_j= t_{\vpi(j)}-t_1, \,\,j\neq \nu$$
hence $t_i=s+\ka_{\vpi^{-1}(i)}
$, $i=1,\dots, d$, and thus $t_1=s$.
Note that
with this identification
\begin{multline*}
|V(s, s+\ka_{\vpi^{-1}(2)},\dots, s+\ka_{\vpi^{-1}(d)})|\\=
\prod_{i=2}^{d} |\ka_{\vpi^{-1}(i)}| \prod_{2\le i<j\le d}
 |\ka_{\vpi^{-1}(j)}-\ka_{\vpi^{-1}(i)}|
= \prod_{i\neq \nu}|\ka_i|
\prod_{\substack{1\le i<j\le d\\i\neq\nu,\,j\neq\nu}} |\kappa_j-\kappa_i|
\end{multline*}
and $\sum_{j=1}^d\gamma(s+\kappa_{\vpi^{-1}(j)})=\gamma(s)+\sum_{i\neq\nu} \gamma(s+\kappa_i)$.

If  $I=[a_L,a_R]$ we define $I^\ka=[a_L-\ka_1, a_R-\ka_d]$
and 
\begin{multline*}
\sD_l^\nu=\big\{\kappa: \kappa_1<\kappa_2<\dots< \kappa_d, \,\kappa_\nu=0,\, \kappa_d-\kappa_1\le a_R-a_L,\\
2^{-l-1}\le \prod_{i\neq \nu}|\kappa_i| 
\prod_{\substack{1\le i<j\le d\\i\neq\nu,\,j\neq\nu}} |\kappa_j-\kappa_i|
<2^{-l}\big\}.\end{multline*}
Let $dm\ci{H_\nu} $ denote 
Lebesgue measure on $H_\nu$ (in $d-1$ dimensions). Let, 
 for $\ka\in H_\nu$, denote by $\ga_\ka$ the offspring curve 
$\gamma_\ka(s)=\gamma(s)+\sum_{i\neq\nu}\gamma(s+\ka_i)=
\sum_{i=1}^d\gamma(s+\ka_i)$.
Then
\begin{multline}\label{Mlpidescr}
M_l^\vpi [f_1,\dots, f_d](x)\,=\,
\int_{\sD^\nu_l} \int_{I^\ka} \exp(-\ic \inn{x}{\gamma_\ka(s)})
 \,\times
\\f_{\vpi(\nu)}(s) w(s) \prod_{i\neq \nu}[f_{\vpi(i)}(s+\ka_i)w(s+\ka_i)] 
\,ds\,
dm\ci{H_\nu}(\ka)\,.
\end{multline}

Let $$w_\kappa(t)= |\tau_{\gamma_\kappa}(t)|^{\frac{2}{d^2+d}},$$
the weight  for the 
 affine arclength measure associated to $\gamma_\kappa$.
 By assumption the offspring curve 
$\gamma_\kappa$ is (after possible reparametrization)  an affine image of a curve in the family $\fC$.
Thus, by affine invariance  and invariance under reparametrizations
 we have the inequality
$$\Big\|
\int_{I_\kappa}
 \exp(-\ic \inn{\cdot }{\gamma_{\ka}(s)}) g(s) w_{\kappa}(s)
\,ds \Big\|_{L^{Q,\infty}}\le \cB \|g\|_{L^Q(w_\kappa)}.
$$
By hypothesis  \eqref{torsionoffspringstr}, and
$\sum_{i=1}^d \eta_{\vpi(i)}=\sum_{i=1}^d \eta_{i}=1$,
$$
w_{\kappa}(s)\ge \fc_2^{\frac{2}{d^2+d}} 
\prod_{i=1}^d w(s+\kappa_i)^{\eta_{\vpi(i)}}, \quad s\in I^\ka\,.
$$
By Observation \ref{observationweight} we then also have
\begin{multline*}
\Big\|
\int_{I^\ka}
 \exp(-\ic \inn{\cdot }{\gamma_{\ka}(s)}) g(s) \fc_2^{\frac{2}{d^2+d}} 
\prod_{i=1}^d w(s+\kappa_i)^{\eta_{\vpi(i)}}
ds \Big\|_{L^{Q,\infty}}\\
\lc \cB 
\Big(\int_{I_\kappa}|g(s)|^Q  \fc_2^{\frac{2}{d^2+d}} 
\prod_{i=1}^d w(s+\kappa_{i})^{\eta_{\vpi(i)}} ds\Big)^{1/Q}\,.
\end{multline*}
We  apply this with $$g(s)\equiv G^{\kappa}(s)
:= 
f_{\vpi(\nu)}(s) w(s)^{1-\eta_{\vpi(\nu)}} \prod_{i\neq \nu}
[f_{\vpi(i)}(s+\ka_i)w(s+\ka_i)^{1-\eta_{\vpi(i)} }]
$$
and, using the relation $\frac{2}{d^2+d}(\frac 1Q-1)=-\frac{1}{Q}$, we  arrive at
\begin{multline}\label{smallerweight}
\Big\|
\int_{I^\ka}
 \exp(-\ic \inn{\cdot }{\gamma_{\ka}(s)}) G^\ka(s) 
\prod_{i=1}^d w(s+\kappa_i)^{\eta_{\vpi(i)}}
ds \Big\|_{L^{Q,\infty}}\\
\lc \cB \fc_2^{-1/Q} 
\Big(\int_{I_\kappa}|G^\ka(s)|^Q   
\prod_{i=1}^d w(s+\kappa_i)^{\eta_{\vpi(i)}} ds\Big)^{1/Q}\,.
\end{multline}
Now use the  
triangle inequality for an equivalent norm 
in the   space $L^{Q,\infty}$  and apply the analogue of the integral Minkowski
 inequality to get 
\begin{align*}
&\big\| M_{l}^\pi[ f_1, \cdots, f_d]\|_{L^{Q,\infty}}\\
&\lc  \int_{\sD_l^\nu} \Big\|  \int_{I^\ka} e^{\ic \inn{\cdot}{\gamma_\ka(s)}} G^\ka(s) 
\prod_{i=1}^d  w(s+ \kappa_i)^{\eta_{\vpi(i)}} \, ds \Big\|_{L^{Q,\infty}} \,
 dm\ci{H^\nu}(\ka)
\\
&\lc  \fc_2^{-1/Q} \cB \int_{\sD_l^\nu} \Big(  \int_{I^\ka}\big|
f_{\vpi(\nu)}(s) w(s)^{1-\eta_{\vpi(\nu)}}
\prod_{i\neq \nu} [f_{\vpi(i)}(s+\ka_i) w(s+\ka_i)^{1-\eta_{\vpi(i)}}]
\big|^Q \\
&\qquad\qquad\qquad\qquad\qquad \times
\prod_{i=1}^d w(s+\ka_i)^{\eta_{\vpi(i)}} \,ds\Big)^{1/Q} 
 dm\ci{H^\nu}(\ka)
\end{align*}
which is
\begin{multline*}
\lc  \fc_2^{-1/Q} \cB \int_{\sD_l^\nu} \Big(  
\int|f_{\vpi(\nu)}(s) w(s)^{1-\eta_{\vpi(\nu)}/Q'}|^Q ds\Big)^{1/Q} \\ \times 
\prod_{i\neq\nu} \|f_{\vpi(i)} w^{1-\eta_{\vpi(i)}/Q'}\|_\infty
 dm\ci{H^\nu}(\ka)\,.
\end{multline*}
By \eqref{sublevelsetest} we have $m\ci{H^\nu}(\sD_l^\nu)\lc 2^{-2l/d}$.
Thus the last displayed quantity is less than a constant times 
\begin{align*}
&  2^{-2l/d}  \fc_2^{-1/Q} \cB 
\big\|f_{\vpi(\nu)} w^{1-\eta_{\vpi(\nu)}/Q'}\big\|_Q
\prod_{i\neq\nu} \|f_{\vpi(i)} w^{1-\eta_{\vpi(i)}/Q'}\|_\infty
\\
=&2^{-2l/d}  \fc_2^{-1/Q} \cB 
\big\|f_{1} w^{1-\eta_d/Q'}\big\|_Q
\prod_{j=2}^{d} \|f_{j} w^{1-\eta_{j}/Q'}\|_\infty
\end{align*}
and \eqref{MlLQinftyunbalanced}  is proved.
\end{proof}

\begin{lemma}
\label{Mlemma}
 Let $q_i\in [Q,\infty]$, $\rho_i\in [2,\infty]$, $\vth_i\in [0,1]$
 satisfy 
$$\sum_{i=1}^d \Big(\frac{1}{q_i}, \frac 1{\rho_i}, \eta_i\Big) = \Big(\frac 1Q, \frac 12, 1 \Big).$$
Let
\begin{align}
\label{pidef}
\frac{1}{p_i}&= \frac{d-2}{d+2} \frac{1}{q_i}+ \frac{4}{d+2} \frac 1{\rho_i}\,,
\\
\label{betaidef}\beta_i &= \frac{d-2}{d+2} \Big(1-\frac{\eta_i}{Q'}\Big)
+ \frac{3-d}{d+2}\,.
\end{align}
Then $\sum_{i=1}^d p_i^{-1}=d/Q$, $\sum_{i=1}^d \beta_i=d/Q$, and we have
\Be\label{Mest}
\big\|  \cM[f_1,\dots, f_d]\big\|_{L^{Q/d,\infty}}
\lc  \cC
\cB^{\frac{d-2}{d+2}}
\prod_{i=1}^d \|f_i\|_{b^1_{\beta_i}(w,L^{p_i,1})}
\Ee
with
\Be \label{cCdef}
\cC=  \big(N_1/\fc_1\big)^{\frac{d-2}{2(d+2)}}
\fc_2^{-\frac{4}{(d+2)Q}}.\Ee
\end{lemma}

\begin{proof}
We may interpolate the $L^2$ bounds and the $L^{Q,\infty}$ bounds for $\cM_l$ to get a satisfactory 
estimate for the 
$L^{Q/d,\infty}$ norm of each $\cM_l$ but the resulting estimates cannot be summed in $l$. We use a familiar 
trick  from \cite{Bo1},  estimating  $\sum_l \cM_l$ using the bound \eqref{MlL2} for $2^l\le \Lambda$
and  the bound \eqref{MlLQinfty} for $2^l> \Lambda$, for $\Lambda$ to be determined.
For fixed $\alpha>0$ we need to estimate the measure of 
$G_\alpha= \{x: |\cM_l[f_1,\dots, f_d]|>2\alpha\}$. By Tshebyshev's inequality, 
$$\meas(G_\alpha)
\le \alpha^{-2}\big\|\sum_{2^l\le \Lambda}\cM_l[f_1,\dots, f_d]\big\|_2^2
+ \alpha^{-Q}\|\sum_{2^l> \Lambda}\cM_l[f_1,\dots, f_d]\|_{L^{Q,\infty}}^Q$$
and applying Lemma \ref{Mlestimates} we obtain
\Be \label{Lambdaestimate} \meas(G_\alpha)
 \le \alpha^{-2} \La^{\frac{d-2}{d}} \Gamma^2
+ \alpha^{-Q}\La^{-2Q/d} \Delta^Q\, ,
\Ee
with $$
\Gamma:= (N_1/\fc_1)^{1/2} \prod_{i=1}^d
\big\|f_i w^{\frac{3-d}{4}}\big\|_{\rho_i}, \qquad 
\Delta:=\cB \fc_2^{-1/Q} \prod_{i=1}^d\big\|f_i w^{1- \frac{\eta_i}{Q'}}\big\|_{q_i}.
$$
We choose $\Lambda$ so that the two expressions on the right hand side
 of \eqref{Lambdaestimate} balance,
i.e. $\Lambda= (\alpha^{2-Q}\Gamma^Q/\Delta^2)^{\frac d{d-2+2Q}}$.
This leads to the bound
$ \meas(G_\alpha)\lc \big(\alpha^{-1}\Delta^{\frac{d-2}{d+2}}\Gamma^{\frac 4{d+2}}\big)^{\frac{(d+2)Q}{d-2+2Q}}$
and we have $\frac{(d+2)Q}{d-2+2Q}=\frac Qd$ for $Q=\frac{d^2+d+2}{2}$. Thus
$$
\|\cM[f_1, \cdots, f_d]\|_{L^{Q/d,\infty}} \lc \cC \cB^{\frac{d-2}{d+2}}  \Delta^{\frac{d-2}{d+2}}\Gamma^{\frac 4{d+2}}
$$
with 
$\cC$ as in \eqref{cCdef}.

By Lemma \ref{thetar} the previous display implies that
\Be \label{wktypeQdesr}
\|\cM[f_1, \cdots, f_d]\|_{L^{Q/d,\infty}} \lc \cC \cB^{\frac{d-2}{d+2}} 
\prod_{i=1}^d \| f_i\|_{\overline  Y_{\frac{4}{d+2},1}^i}
\Ee
where the interpolation space refers to the couple $\overline Y^i$
with 
$$Y_0^i=b^1_{1-\frac{\eta_i}{Q'}}(w,L^{q_i})\,,\qquad Y_1^i=b^1_{\frac{3-d}{4}} (w,L^{r_i})\,.$$ 
Since the block Lorentz spaces are retracts of sequence spaces ({\it cf.} the discussion following  \eqref{blockLorentzdef}) the formula
\eqref{obvious} implies
the continuous embedding
$$b^1_{(1-\vth)(1-\frac{\eta_i}{Q'})+\vth\frac{3-d}{4}}((L^{q_i}, L^{\rho_i})_{\vth,1})
\hookrightarrow (b^1_{1-\frac{\eta_i}{Q'}}(w,L^{q_i}), b^1_{\frac{3-d}{4}}(w, L^{\rho_i}))_{\vth,1}
$$
We apply this with $\vth=4/(d+2)$. Then if
$p_i$, $\beta_i$ are  in \eqref{pidef}, \eqref{betaidef} 
the usual interpolation formula for Lorentz spaces gives
$$(L^{q_i}, L^{\rho_i})_{\frac{4}{d+2},1}= L^{p_i,1}\,.$$
and it follows that  $b^1_{\beta_i} (w, L^{p_i,1}) $ is continuously  embedded in $\overline  Y^i_{4/(d+2),1}$. Now  \eqref{Mest}
follows from \eqref{wktypeQdesr}.
\end{proof}

As stated above the conditions \eqref{pidef}, \eqref{betaidef} 
give 
$\sum_{i=1}^d p_i^{-1}=\sum_{i=1}^d\beta_i=d/Q$.
In particular we may choose  $p_i=Q$ and $\beta_i=1/Q$ and this choice 
 yields an estimate
for $f_i \in {b^1_{1/Q}(w,L^{Q,1})}
$, in particular (after setting all $f_i$ equal to $f$)
\Be \label{restrwtest}
\|\cE_w f\|_{L^{Q,\infty} } \lc \cC^{1/d}\cB^{\frac{d-2}{d^2+2d} }
\|f\|_{b^1_{1/Q}(L^{Q,1})}\,.
\Ee
However we need a better estimate for 
$f_i$ in the larger space 
${b^Q_{1/Q}(w,L^{Q})} = L^Q(w)$. In what follows  write
$b^p_s(L^p)=b^p_s(w,L^p)$ as the weight $w$ will be fixed.

\begin{proof}[Proof of Theorem \ref{mainthm}, cont.]
We  choose $n>Q$ and estimate the $n$-linear operator
$$T[f_1,\dots, f_n]= \prod_{i=1}^n \cE_w f_i $$
in $L^{r,\infty}$ where $r=Q/n<1$.

For every  permutation $\vpi$ on $n$ letters we may 
write  $$T[f_1,\dots, f_n]= \cM[f_{\vpi(1)},\dots, f_{\vpi(d)}]
\prod_{i=d+1}^n \cE_w f_{\vpi(i)}\,.$$
Notice that by H\"older's inequality for Lorentz spaces
\Be \label{HoelderLor}
\big\|T[f_1,\dots, f_n]\|_{L^{r,\infty}}
\le C \| \cM[f_{\vpi(1)},\dots, f_{\vpi(d)}]\|_{L^{Q/d,\infty}} \prod_{i=d+1}^n
\|\cE_w f_{\vpi(i)}\|_{L^{Q,\infty}}\,.
\Ee

We apply 
Lemma \ref{Mlemma} for  special choices of the parameters $q_i$, $\rho_i$.
Let $\mu$ be a small parameter (say $|\mu|\ll (10 Qd^2)^{-1}$), put
  $\rho_i= 2d$, $i=1,\dots, d$,
 let
\begin{align*}
\frac{1}{q_3}&= \frac{1}{Qd}+\mu \frac{d+2}{d-2}\,,
\\
\frac{1}{q_2}&= \frac{1}{Qd}+ \mu\frac{d+2}{n-2}\,,
\\
 \frac{1}{q_1}&= \frac{1}{Qd}-\mu (d+2)\frac{n-1}{n-2}\,,
\end{align*}
and set $q_d=\dots=q_3$.
Then $\sum_{i=1}^d q_i^{-1}=Q^{-1}$
and $q_3<q_2<Qd<q_1$ if $\mu>0$ (for $\mu<0$ these inequalities are reversed).
Now by \eqref{pidef}, $p_i^{-1}=\frac{d-2}{d+2}q_i^{-1}+ \frac{2}{d(d+2)}$ and since $Q=\frac{d^2+d+2}{2}$ we  have $\frac{d-2}{(d+2)Q}+\frac{2}{d+2}=\frac dQ$ and thus 
$p_i^{-1}=Q^{-1} +(q_i^{-1}- (Qd)^{-1})\frac{d-2}{d+2}$. Hence
$p_3=\dots=p_d$ and
\begin{align*}
\frac{1}{p_3}&= \frac{1}{Q}+\mu \,,
\\
\frac{1}{p_2}&= \frac{1}{Q}+\mu \frac{d-2}{n-2}\,,
\\
 \frac{1}{p_1}&= \frac{1}{Q}-\mu (d-2)\frac{n-1}{n-2}\,.
\end{align*}
Then $\sum_{i=1}^d p_i^{-1}= d/Q$, moreover $p_3<p_2<Q<p_1$ if $\mu>0$.
A crucial property of our choices is 
\Be
\label{nminustwocond} \frac{1}{p_2}= \frac{d-2}{n-2}\frac 1{p_3}+
 \frac{n-d}{n-2}\frac 1{Q}.
\Ee

 Let $\beta_i$ be as in 
\eqref{betaidef} (with $\eta_3=\dots=\eta_d$  and the choice of $\eta_2$ and $\eta_3$ to be determined later).
With these choices 
we use  \eqref{HoelderLor}, and apply \eqref{Mest}
 for the term involving $\cM$ and    \eqref{restrwtest} for the remaining $n-d$ terms. This results in 
\begin{multline*}
\big\|T[f_1,\dots, f_n]\|_{L^{r,\infty}}\lc 
\cB^{\frac{d-2}{d+2}}
\|f_{\vpi(1)}\|_{b^1_{\beta_1}(L^{p_1,1})}
\|f_{\vpi(2)}\|_{b^1_{\beta_2}(L^{p_2,1})}\\ \times
\prod_{i=3}^d \|f_{\vpi(i)}\|_{b^1_{\beta_3}(L^{p_3,1})}
\prod_{j=d+1}^n [\cC^{1/d}\cB^{\frac{d-2}{d^2+2d}} \|f_{\vpi(j)}\|_{b^1_{1/Q}(L^{Q,1})}]\,.
\end{multline*}
Now fix the first two entries and take generalized geometric means of these estimates to get
\begin{multline*}\big\|T[f_1,\dots, f_n]\|_{L^{r,\infty}}\lc 
\cC \cB^{n-d+\frac{d-2}{d+2}}
\|f_{\vpi(1)}\|_{b^1_{\beta_1}(L^{p_1,1})}
\|f_{\vpi(2)}\|_{b^1_{\beta_2}(L^{p_2,1})}\\ \times
\prod_{i=3}^n\big[ \|f_{\vpi(i)}\|_{b^1_{\beta_3}(L^{p_3,1})}^{\frac{d-2}{n-2}}
 \|f_{\vpi(i)}\|_{b^1_{1/Q}(L^{Q,1})}^{\frac{n-d}{n-2}}].
\end{multline*}
By  \eqref{nminustwocond},
and Lemma \ref{thetar}, we get
\begin{multline*}\big\|T[f_1,\dots, f_n]\|_{L^{r,\infty}}\lc 
\cC \cB^{\frac{d-2}{d+2}}
\big(\cC^{\frac 1d}\cB^{\frac{d-2}{d^2+2d}}\big)^{n-d}
\|f_{\vpi(1)}\|_{b^1_{\beta_1}(L^{p_1,1})}
\|f_{\vpi(2)}\|_{b^1_{\beta_2}(L^{p_2,1})}\\
\times \prod_{i=3}^n
\|f_{\vpi(i)}\|_{(b^1_{\beta_3}(L^{p_3,1}),
 b^1_{1/Q}(L^{Q,1}))_{\frac{n-d}{n-2},r}} \,.
\end{multline*}
The constant simplifies to $(\cC \cB^{\frac{d-2}{d+2}})^{n/d}$.
By Lemma \ref{quasi} and a trivial embedding for the first two
 factors we also get
\begin{multline*}
\big\|T[f_1,\dots, f_n]\big\|_{L^{r,\infty}}\lc \\
(\cC \cB^{\frac{d-2}{d+2}})^{n/d}
\|f_{\vpi(1)}\|_{b^r_{\delta_1}(L^{p_1,r})}
\|f_{\vpi(2)}\|_{b^r_{\delta_2}(L^{p_2,r})}
\prod_{i=3}^n
\|f_{\vpi(i)}\|_{b^r_{\delta_3}(L^{p_2,r})} \,
\end{multline*}
where $\delta_1=\beta_1$, $\delta_2=\beta_2$ and 
$$\delta_3= \frac{d-2}{n-2} \beta_3+ \frac{n-d}{n-2} \frac{1}{Q}\,.$$
We may choose $\eta_2, \eta_3$, so that $\delta_2\neq\delta_3$.
This is needed  for the  application of Theorem \ref{interpolationtheorem}.
We choose $X_0=L^{p_2,1}$, $X_1=L^{p_1,1}$ and by the conclusion \eqref{Tconclbalanced} of that theorem we obtain 
$$\big\|T[f_1,\dots, f_n]\|_{L^{r,\infty}}\lc 
(\cC \cB^{\frac{d-2}{d+2}})^{n/d}\prod_{i=1}^n \|f_i\|_{b^{nr}_s((X_0,X_1)_{1/n, nr})},$$ with
$s= \frac{1}{n}(\delta_1+\delta_2+(n-2)\delta_3)$. 
Now $r=Q/n$  and $s=1/Q$ since
$$\begin{aligned} sn&=\sum_{i=1}^n\delta_i= \delta_1+\delta_2+(n-2)\Big(
\frac{d-2}{n-2}\beta_3+ \frac{n-d}{n-2}\frac{1}{Q}\Big)
\\&=\beta_1+\beta_2+ (d-2)\beta_3+\frac{n-d}{Q}
\\&=\sum_{i=1}^d \Big(\frac{d-2}{d+2}(1-\frac{\eta_i}{Q'}) +\frac{3-d}{d+2}\Big)
+\frac{n-d}{Q}
\\&= \frac{d-2}{d+2}(d-1+\frac{1}{Q})+ \frac{d(3-d)}{d+2}+ \frac{n-d}Q= \frac{d}{Q}+\frac{n-d}{Q}.
\end{aligned}$$
Also  $\frac{n-1}{n}\frac 1{p_2}+\frac 1n\frac 1{p_1}=\frac{1}{Q}$
and thus 
$(X_0,X_1)_{1/n, nr}=(L^{p_2,r}, L^{p_1,r})_{1/n,nr}=L^Q$, and therefore
$b^{nr}_s(w,(X_0,X_1)_{1/n, nr})= b^Q_{1/Q}(w,L^Q)=L^Q(w)$.

Take $f_1=\dots=f_n=f$ and since
$$\big\|T[f,\dots, f]\big\|_{L^{r,\infty}} \approx \big\|\cE_w f\big\|_{L^{Q,\infty}}^n$$
we get  $\cB^n \lc (\cC \cB^{\frac{d-2}{d+2}})^{n/d}$ 
or $\cB \lc \cC^{1/d} \cB^{\frac{d-2}{d(d+2)}}$, and this finally implies
\Be\label{finalbound} 
\cB\le C(d) \cC^{\frac{d+2}{d^2+d+2}}\Ee
where 
$\cC$ is as in \eqref{cCdef}.
\end{proof}

\section{Proof of Theorem \ref{powerthm}}\label{powertheoremsect}
The crucial idea, due to Drury and Marshall \cite{DM2}, is to use an 
exponential parametrization.
Fix $R>0$,  $I_R=[0,R]$ let 
$b=(b_1,\dots, b_d)\in \bbR^d$, $b_i\neq 0$,  
$$\gamma(t)\equiv \Gamma^b(t) =(b_1^{-1} e^{b_1 t},\dots, b_d^{-1} e^{b_n t})\,,$$
and let $\fC_{b,R}$ be the class consisting of $\Gamma^b$ and  restrictions of 
$\Gamma^b$ to subintervals.
The objective is to prove the 
bound $\|\cE_wf\|_{L^{Q,\infty}} \le C (\int_0^R|f(t)|^Qw(t) dt)^{1/Q}$ with a constant independent of 
$b$ and  $R$. This inequality is trivial if some of the $b_i$
coincide since then $w=0$ and thus $\cE_w=0$.
When
the $b_i$ are pairwise different
a priori we at least know that the quantity $\cB(\fC_{b,R})$ is  finite, but  
with a bound possibly depending on 
 $b$ and $R$. To see this
one may apply the result of \cite{bos1} since the torsion $\tau$ is positive
and  $\Gamma^b$ is smooth on the compact interval $I_R$.

We need to check  Hypotheses \ref{hypess}. Most of this work has already been done in \cite{DM2}.
If we form the $\kappa$-offspring curves  $\gamma_\kappa$ (see the definition in \S\ref{hyp}), then $\gamma_\ka(t)= \gamma(t) E(\kappa)$ where 
$E(\kappa)$ denotes the diagonal matrix with entries 
$$E_{ii}(\kappa)=
\sum_{j=1}^de^{b_i\kappa_j}$$ and thus is an affine image of a curve in $\fC$.
The bounded multiplicity hypothesis  is valid by the  discussion in \cite{DM2}, p. 549
(this goes back to a paper by Steinig \cite{steinig}).
The crucial inequality \eqref{firstmainhyp} has already been verified by Drury and Marshall 
who proved the relevant \lq total positivity\rq  \ bound in \cite {DM2}, p. 546; \cf.  also  \cite{DeM} for an alternative approach. The constant $\fc_1$ is independent of $b$ and the estimate holds globally.

It remains to verify the second main assumption,
inequality \eqref{torsionoffspringstr}.
We recall from \cite{DM2}, \cite{bos1} formulas for the torsion
 $\tau(t) = \tau_b(t)$ of $\Gamma^b
(t)$:
$$ |\tau_b(t)| = |V(b_1, \dots, b_d)| \, \exp \Big(t \sum_{j=1}^d b_j \Big) .$$
where $ V(b_1, \dots, b_d) = \prod_{1\le i<j\le d}(b_j-b_i)$
is the Vandermonde determinant.
For the  torsion of the offspring curve
$\gamma_\kappa$ we have 
$$ |\tau_{\gamma_\ka}(t)| = |V(b_1, \dots, b_d)| \, \exp \Big(t \sum_{j=1}^d b_j \Big) \prod_{i=1}^d
E_{ii} (\ka)\,.$$
Since
$E_{ii}(\ka) \ge \exp (b_i \kappa_j)$, for $1\le j\le d$, we have
$$ \prod_{i=1}^d E_{ii}(\ka) \ge \prod_{i=1}^d \exp (b_i \kappa_j) = \exp
\Big(\kappa_j \sum_{i=1}^d b_i \Big) .$$
Therefore, it follows that $$|\tau_{\gamma_{\ka}}(t)| \ge |\tau (t+\kappa_j)|,
\quad 1\le j\le d,$$ and  
\eqref{torsionoffspringstr} is proved with
$\fc_2=1$.
Now Theorem \ref{mainthm} gives a uniform bound for the classes $\fC_{b,R}$ and 
letting  $R\to \infty$, we also get a global result.
To prove the asserted result for  monomial curves on $[0,\infty)$ we consider the intervals $[0,1]$ and $[1,\infty)$ separately, introduce an exponential parametrization on each interval and use the invariance of affine 
arclength measure under changes of parametrizations. \qed

\medskip

\noi{\it Remark.}
The sharp $L^p\to L^q$ estimates for monomial curves in our earlier paper 
\cite{bos1} 
have been recently  extended by Dendrinos and M\"uller \cite{DeM} to cover small local perturbations of monomial curves. In their setting 
they prove an analogue of 
the geometric assumption \eqref{firstmainhyp}; moreover, Lemma 2 and Lemma 4
 in \cite{DeM} show that a variant of the above calculation 
remains true and 
\eqref{torsionoffspringstr} continues to hold (although
no  global uniformity result is proved in this setting). One  can thus 
obtain  
a local analogue of Theorem \ref{powerthm} for perturbations of monomial curves. As a consequence one also gets an  $L^{p_d,1}\to L^{p_d}$ endpoint 
result for  
every curve of finite type, defined on a compact interval, and the estimate is stable under small perturbations.

\section{Curves of simple type and the proof of Theorem \ref{polthm}}\label{simpletypethms} 
As observed in \cite{DM1} some technical issues in  the restriction
problem with respect to affine measure become  easier  for classes 
of curves of {\it simple type} on some interval $I$, namely $\gamma\in C^d$ is supposed to be of the form
\begin{equation}\label{simple}
\gamma (t)=\Bigl( t,\frac{t^2}{2!},\dots
,\frac{t^{d-1}}{(d-1)!},\phi (t) \Bigr) , \quad t\in I.
\end{equation}
In this case $\tau(t)= \phi^{(d)}(t)$.
Moreover, because of the triangular structure of the matrix defining the torsion,  the torsion
 of the offspring curve is easy to compute. We get
$$\tau_{\gamma_\kappa}(t)= d^{d-1} \Big(\sum_{j=1}^d\phi^{(d)}(t+\kappa_j)\Big)\,.$$
Consequently,  the  verification of  condition \eqref{torsionoffspringstr}
is often trivial:

\begin{observation} \label{simpleobs} Let $\gamma$ be as in \eqref{simple} and assume that on an interval $I$ the function $\phi^{(d)}$ is of constant sign.
Let $t+\kappa_1,\,t+\kappa_d\in I$.
Then condition \eqref{torsionoffspringstr} holds with $\fc_2=1$.
\end{observation}
Indeed, for $1\le j\le d$,
$$|\tau_{\gamma_\ka}(t)|= d^{d-1}
\Big|\sum_{i=1}^d \phi^{(d)}(t+\kappa_i)\Big|
\ge d^{d-1} \,|\phi^{(d)}(t+\kappa_j)|= d^{d-1}|\tau (t+\kappa_j)|$$
so that  \eqref{torsionoffspringstr} holds.

In contrast, the  verification of our first main hypothesis
\eqref{firstmainhyp}
can be  hard. The inequality on suitable subintervals 
has been verified for polynomial curves $(P_1,\dots, P_d)$ by Dendrinos
and Wright \cite{DeW}, and their  argument is of great  complexity. 
An extension  to curves whose coordinate functions are rational has been 
worked out in \cite{DeFW}.
Below we give a rather short argument of \eqref{firstmainhyp} 
for the case of a polynomial curve of simple type. In this case
one can prove an estimate which is slightly stronger than \eqref{firstmainhyp}.

We finally remark that both \eqref{firstmainhyp} 
and (by the observation above)
\eqref{torsionoffspringstr} hold for a class of \lq convex\rq \  curves of simple type considered in \cite{bos2}. This  class also contains nontrivial  examples in which  the curvature vanishes to infinite order at a point.

\subsection*{Jacobian estimate for polynomial curves of simple type}
The strengthened version of 
\eqref{firstmainhyp}  
 for simple  polynomial curves is

\begin{proposition}\label{jacest}
Let $\gamma(t)= 
\big(t, \frac{t^2}{2!},\dots,
\frac{t^{d-1}}{(d-1)!} ,P_b(t)\big)$, $P_b(t)=\sum_{j=0}^Nb_j t^j$.
Put
$$J(t,\ka)=|\det (\gamma '
(t+\kappa_1), \cdots, \gamma ' (t+\kappa_d ))| $$
where $\kappa_1<\cdots< \kappa_d$.
Then $\bbR$ is
the union of $C(N,d)$ intervals $I_n$ such that whenever $t+\ka_1$, $t+
\ka_d \in I_n$, we have
\begin{equation}\label{ineq1}
J(t,\kappa)\geq c(N,d)\,|V(\ka)|\,\max\{|\phi ^{(d)}(t+\kappa_j)|: \, 1\le
j\le d \};
\end{equation}
here $V(\ka)= \prod_{1\le i<j\le d}  (\kappa_j-\kappa_i)$ and $c(N,d)>0$.
\end{proposition}

We begin by proving an auxiliary lemma where the polynomial assumption is not used.

\begin{lemma}\label{Psiest} 
Let  $\phi \in C^d(\bbR)$ and let 
 $J_d (s_1,\dots,s_d; \phi)$ denote the determinant of
the $d\times d$ matrix with the $j$-th column 
$(1, s_j,\dots, s_j^{d-2}/(d-2)!, \phi'(s_j))^T$. Then for 
$-\infty < s_1 < \cdots < s_d < \infty$,
\begin{equation} \label{integ}
J_d (s_1,\dots,s_d; \phi) = \int_{s_1}^{s_d} \phi^{(d)}(u)
\Psi(u; s_1,\dots,s_d) du
\end{equation}
%
where $\Psi\equiv \Psi_d$ satisfies
\Be\label{psiestimate} 0\le \Psi(u; s_1,\dots, s_d) \le \Big|\frac{V(s_1, \dots, s_d)} {s_d-s_1}\Big| 
\quad \text{ for all }  u \in [s_1, s_d]. \Ee
%
\end{lemma}
\begin{proof} 
 We
will follow the arguments in \cite{bos2}.
We first show that
\Be\label{JdfromJd-1}
J_d (s_1,\dots,s_d; \phi) \,=\,\int_{s_1}^{s_2}\cdots\int_{s_{d-1}}^{s_d} J_{d-1}(\sigma_1,\dots, \sigma_{d-1}; \phi')\,
d\sigma_{d-1}\cdots d\sigma_1\,.
\Ee
To prove this we first note 
that 
 since a determinant is
zero if two columns are equal, we have
\begin{align*} \label{}
&J_d (s_1,\dots,s_d; \phi) 
= -\int_{s_1}^{s_2}
\partial_1 J_d (\sigma_1, s_2, \dots,s_d; \phi) d\sigma_1
\\&=(-1)^{d-1} \int_{s_1}^{s_2}\cdots \int_{s_{d-1}}^{s_d}
\partial_{d-1}\cdots \partial_1 J_d (\sigma_1, \dots, \sigma_{d-1}, s_d; \phi)
\, d\sigma_{d-1}\cdots d\sigma_1\,.
\end{align*}
Now $\partial_{d-1}\cdots \partial_1 J_d (\sigma_1, \dots,
\sigma_{d-1}, s_d; \phi)$ is the determinant of a matrix with the
first row $(0,\dots,0,1)$, and one easily checks that
$$ \partial_{d-1}\cdots \partial_1 J_d (\sigma_1, \dots, \sigma_{d-1}, s_d; \phi)
= (-1)^{d-1} J_{d-1} (\sigma_1, \dots, \sigma_{d-1}; \phi') .$$
Combining the two previous displays yields \eqref{JdfromJd-1}.

We now wish to iterate this formula.
It is convenient to denote by  $x^m = (x^m_1,\dots,x_m^m)$ 
 a point in $\bbR^m$ with
$x^m_1\le\dots\le x_m^m$ (i.e with increasing coordinates). We shall set $(s_1,\dots, s_d)=(x_1^d,\dots,x^d_d)=x^d$.
For $1\le k\le d-2$ define
$$\sH_{d-k}(x^{d-k+1})= 
\{ x^{d-k}\in \bbR^{d-k}: \, x_j^{d-k+1} \le
x_j^{d-k}\le x_{j+1}^{d-k+1}, ~~ 1\le j\le d-k \}.$$
Note that  if the coordinates of $x^{d-k+1}_j$ are increasing in $j$ then for every 
$x^{d-k}\in \sH_{d-k}$ the coordinates of $x^{d-k}$ are increasing.
The formula  \eqref{JdfromJd-1}
can be rewritten as
$$ J_d (x^d; \phi) = \int_{\sH_{d-1}(x^d)}  J_{d-1} (x^{d-1} ;
\phi') \,dx^{d-1} . $$
Induction  gives, for $1\le k\le d-2$
$$ J_d (x^d; \phi) = \int_{\sH_{d-1}(x^d)} \cdots \int_{\sH_{d-k} (x^{d-k+1})}
J_{d-k}(x^{d-k+1}; \phi^{(k)}) \, dx^{d-k}\cdots dx^{d-1}\,.
$$
We also have
$$ J_2(x^2; \phi^{(d-2)}) = \phi^{(d-1)}(x^2_2)- \phi^{(d-1)}(x^2_1)=
\int_{x_1^2}^{x_2^2} \phi^{(d)}(u) \,du\,.$$
Hence, if 
\begin{multline*}
\sG_u(x^d)= \{ (x^2, x^3, \cdots, x^{d-1}): \\
x^2_1 \le u\le x^2_2,\, x^{d-k} \in \sH_{d-k}(x^{d-k+1}), \, 1\le k\le d-2\}
\end{multline*}
and 
$$ \Psi(u; x^d): = \int_{\sG_u(x^d)} dx^2 dx^3 \cdots dx^{d-1}$$
we get 
$$J_d(x^d;\phi) = \int_{x_1^d}^{x_d^d} \phi^{(d)}(u) \Psi(u;x^d) \,du\,.$$
If $x^d=(s_1,\dots, s_d)$ this is \eqref{integ}.

Observe that, for each $u \in [x_1^d, x_d^d]$, the set  $\sG_u(x^d) $ is contained in
the rectangular box
$$ B_{2}(x^d)\times  \cdots \times B_{d-1}(x^d) $$
where
$$ B_{d-k}(x^d) = \{ x^{d-k}\in \bbR^{d-k}: ~ x_j^d \le
x_j^{d-k}\le x_{j+k}^d, ~~ 1\le j\le d-k \} .$$
Since
$$ \vol_{d-k}(B_{d-k}) = \prod_{j=1}^{d-k} (x_{j+k}^d - x_j^d) $$
it follows by rearranging the factors that
\begin{align*}
 \Psi&(u;x^d) = 
\meas(\sG_u(x^d)) 
\le \prod_{k=2}^{d-1} \vol_k(B_k) \\
&= \prod_{2\le i \le d-1}(x_d^d -x_i^d)  \prod_{1\le i<j\le d-1}(x_j^d-x_i^d)
= (x_d^d-x_1^d)^{-1} V(x_1^d, \cdots, x_d^d) .
 \end{align*}
This proves  \eqref{psiestimate}.
\end{proof}

We also need the following observation on polynomials.
\begin{lemma} \label{polynomialsublevelest}
Let  $p$ be  a real-valued
polynomial of degree $\leq N$ and $|p(t)|>0$ on $(a,b)$.
Then, for every $\varepsilon \in (0, 2^{-N})$,
\begin{equation}\label{pest}
\big|\{ t\in (a,b):|p(t)|<\varepsilon |p(b)|\}\big| \leq 2N \varepsilon
^{\frac{1}{2N} }(b-a).
\end{equation}
\end{lemma}\begin{proof}
To show \eqref{pest} we check  that for $c\in \bbR$ and
$0<\delta <1/2$ we have
\begin{equation}\label{pest3}
|\{t\in (a,b): |t-c|< \delta |b-c|\}| \leq 2\delta (b-a).
\end{equation}
If $b<c$ then $|t-c|>|b-c|$ if $t\in [a,b]$, so $  \{ t\in
(a,b):|t-c|<\delta |b-c|\}=\emptyset$ since $\delta <1/2$. If
$a\leq c\leq b$ then
$$ |\{ t\in (a,b):|t-c|<\delta |b-c|\}| \leq
2\delta |b-c|\leq 2\delta (b-a).
$$
If $c<a<b$ and $|a-c|\leq b-a$ then
$$
|\{t\in (a,b):|t-c|<\delta |b-c|\}
| \leq \delta |b-c|\leq 2\delta
(b-a).
$$
And if $c<a<b$ and $|a-c|>b-a$, then if $t\in [a,b]$ we have
$|t-c|\geq |a-c|\geq (|b-a|+|a-c|)/2=|b-c|/2$, so $ \{t\in
(a,b):|t-c|<\delta |b-c|\}=\emptyset$ since $\delta <1/2$. This
gives \eqref{pest3}.

Moving towards \eqref{pest}, we may normalize the leading coefficient  and write
$p(t)=\prod_{i=1}^{N_1}p_i(t) \prod_{j=1}^{N_2} q_j(t)$
where $p_i(t)=t-c_i$, $q_j(t) = (t-d_j )^2 +e_j^2$, $c_i, d_j, e_j\in \bbR$, and 
 $N_1+2N_2\leq N$. To establish \eqref{pest} 
we show that
\begin{equation}\label{pest2}
|\{ t\in (a,b):|q(t)|<\delta |q(b)|\}| \leq 2\sqrt{\delta} (b-a)
\end{equation}
if $0<\delta <1/2$ and $q(t)=t-c$  or $q(t)=(t-c)^2 +d^2$. The
case $q(t)=t-c$ follows from \eqref{pest3}. If $q(t)=(t-c)^2 +d^2$
then
$$
\{t\in (a,b):q(t)<\delta q(b)\}\subset \{t\in [a,b]: |t-c|\leq
\sqrt{\delta}|b-c|\}
$$
so
$$
|\{t\in (a,b):q(t)< \delta q(b)\}| \leq 2\sqrt{\delta}(b-a)
$$
by \eqref{pest3}. 
This gives \eqref{pest2}.
Finally $\{ t\in (a,b):|p(t)|<\varepsilon |p(b)|\}$ 
is contained in the union of the
$N_1$ sets 
$\{ t\in (a,b):|p_i(t)|<\varepsilon^{1/(N_1+N_2)}
|p_i(b)|\}$ 
and the $N_2$ sets 
$\{ t\in (a,b):|q_j(t)|<\varepsilon^{1/(N_1+N_2)} |q_j(b)|\}$ 
and thus, if $\eps<2^{-N_1-N_2}$ 
$$\big|\{ t\in (a,b):|p(t)|<\varepsilon |p(b)|\}\big| \leq 
\big(2N_1 \varepsilon^{\frac{1}{N_1+N_2} }+
2N_2 \varepsilon^{\frac{1}{2(N_1+N_2)} }\big)
(b-a)
$$
This proves \eqref{pest}.
\end{proof}

\begin{proof}[Proof of Proposition \ref{jacest}]
With $\phi=P_b$ fixed choose the $I_n$ such that $\phi ^{(d)}$ and
$\phi^{(d+1)}$ are nonzero on the interior of each $I_n$. 
We assume without loss of generality that
 $\phi ^{(d)},\,\phi^{(d+1)}>0$ on the interior of $I_n$.
If we
put $s_j = t+ \kappa_j$, then it follows by Lemma \ref{Psiest}
that
\begin{equation}\label{ineq2}
J(t,\ka)= |J_d (s_1,\cdots,s_d; \phi)| = \Big| \int_{s_1}^{s_d}
\phi^{(d)}(u) \Psi(u) du \Big|
\end{equation}
for some nonnegative function $\Psi(u)=\Psi(u;s_1,\dots, s_d)$ which satisfies
\begin{equation}\label{PsiV}
 \Psi(u) \le V(s_1,\dots, s_d)
/(s_d-s_1) .
 \end{equation}
Note that $V(s_1,\dots, s_d)=V(\ka_1,\dots,\ka_d)$.
By applying \eqref{integ} with $\phi(t)= t^d/(d!)$, we get
$$ \int_{s_1}^{s_d}\Psi (u)\,du = c_d V(\ka)$$
where $c_d = (2! \cdots (d-1)!)^{-1}$.

To see \eqref{ineq1} we use this fact and \eqref{PsiV}. Thus we
have
\begin{align}\label{ineq6}
\int_{[s_1, s_d ]\setminus E}\Psi (u)\,du &= \int_{s_1}^{s_d}\Psi
(u)\,du - \int_{E}\Psi (u)\,du \\
&\geq {c_d \,V(\ka)} - |E| \frac{V(\ka)}{s_d -s_1}
\end{align}
if $E\subset [s_1, s_d ]$. Choose $\varepsilon =\varepsilon (d,N)$
so small that $2N\varepsilon^{\frac{1}{2N}} \leq c_d/2$. Now
assume that $s_1=t+\ka_1\in I_n$, $s_d
= t+\kappa_d \in I_n$. With $$E=\{u\in [s_1, s_d ]: \, \phi
^{(d)}(u)<\varepsilon \phi ^{(d)}(t+\kappa_d )\}$$ we have $|E|<
(s_d -s_1 ) c_d/2$ by Lemma \ref{polynomialsublevelest} and our choice of
$\varepsilon$.

Hence, by \eqref{ineq2} and \eqref{ineq6}, we have
\begin{align*}
J(t,\ka) &\geq \int_{{[s_1, s_d ]\setminus E}}\phi ^{(d)}(u)\,\Psi
(u)\,du \\
&\ge \varepsilon \, \phi ^{(d)}(t+\kappa_d ) \, V(\ka)\, c_d /2 \\
&= \eps \frac{c_d}2\,V(\ka)\, \max\{|\phi ^{(d)}(t+\kappa_j)|: ~ 1\le j\le d \},
\end{align*}
giving \eqref{ineq1} as desired. Here we put $c(N,d)= \varepsilon \,
c_d /2$.
\end{proof}

\begin{proof} [\bf Proof of Theorem \ref{polthm}]
Fix a polynomial $\phi$, of degree $N$. If $d<N$ then the affine arclength measure is identically $0$ and the assertion trivially holds with 
$C(N)=0$. 
Now  assume $N\ge d$. Let $I$ be an interval
on which the inequality  \eqref{ineq1} holds, 
and $\phi^{(d)}$ and $\phi^{(d+1)}$ do   not change sign.
Pick a subinterval $I_0$ on which $\phi^{(d)}$ does not vanish.
Denote by $\fC$ the class of simple curves given by  $\phi$, on $I_0$
or on subintervals of $I_0$.
We notice that the offspring curves
are affine images of the original curves (\cf. the proof of Lemma 3.1 in 
\cite{bos2}). The bounded multiplicity hypothesis  (when the curve is restricted to suitable subintervals)  is  discussed in \cite{DM1}, \cite {DeW}.
Theorem \ref{mainthm} gives the desired conclusion on the interval $I_0$ with no reference to a nondegeneracy 
assumption.   A limiting argument gives the conclusion on the full interval $I$. Since $\phi$ is a polynomial
and by  Proposition \ref{jacest} 
 we have to apply this consideration to only a finite number of intervals.
\end{proof}

\medskip

\noi{\it Remark.} Let $\Gamma(t)= (R_1(t),\dots, R_d(t))$ where $R_i(t)=P_{1,i}(t)/P_{2,i}(t)$, $P_{1,i}, P_{2,i}$ are polynomials.
It  has been proved by Dendrinos, Folch-Gabayet and Wright \cite{DeFW}  that $\bbR$ can be decomposed into a finite 
number of intervals (depending on $d$ and the maximal degree of the polynomials involved) so that 
the crucial hypothesis
 \eqref{firstmainhyp} is satisfied on the interior of each interval. In the special case
of  rational curves of simple type (with $R_i(t)= t^i/i!$, $i=1,\dots, d-1$ and $R_d(t)= P(t)/Q(t)$, $P,Q$ polynomials)  one  can show that after a further  decomposition Observation \ref{simpleobs} applies. Thus 
Theorem
\ref{polthm} extends  to rational curves of simple type.

\section{A note on the range of the sharp $L^p\to L^q$ 
adjoint  restriction theorem for general 
 polynomial curves}\label{Drury-estimate}

Suppose $t\mapsto \gamma (t)=(P_1(t),\dots, P_d(t))$ is a polynomial curve
 in $\bbR^d$,
with the $P_i$  
of degree at most  $n$,
 and suppose that $d\lambda=wdt$ is the affine arclength measure 
on $\gamma$. Dendrinos and Wright \cite{DeW} established the critical  Fourier extension estimate
\begin{equation}\label{ineq1poly}
\|\widehat{fd\lambda}\|_q \leq C(n,p)\,\|f\|_{L^p (\lambda )},\quad
\frac{1}{p}+\frac{d(d+1)}{2q}=1
\end{equation}
in the range 
$1\le p<d+2$
(the range obtained by Christ \cite{Ch} in the nondegenerate case).
Much earlier Drury \cite{D} had proven a restriction estimate for certain curves $(t,t^2 ,t^k)$ in dimension $3$
that was valid for $1\leq p <6$ and therefore valid for some $p$ outside of the Christ range. 
It turns out that by 
replacing two of the estimates in Drury's argument by  estimates of Dendrinos and Wright and of  Dendrinos, Laghi and Wright \cite{DLW}
one can extend Drury's result to general polynomial curves in $\bbR^3$.
Moreover using an  estimate of Stovall \cite{sto},
one can show 
\begin{proposition}
For general polynomial curves in $\bbR^d$, $d\ge 3$, the 
 Fourier extension estimate 
\eqref{ineq1poly} holds  for 
\begin{equation}\label{ineq1.6poly}
1\leq p<d+3+\frac{2(d-3)}{d^2-3d+4}\,.
\end{equation}
\end{proposition}
\begin{proof}
What follows, then, is just Drury's argument run with up-to-date technology.
The required result from \S3 of \cite{DeW} (cf. \cite{Ch} for the nondegenerate case) is the $d$-fold convolution estimate
\begin{equation}\label{ineq2poly}
\begin{aligned}&\|(fd\lambda )\ast \cdots \ast (fd\lambda )
\|_r \leq C(n,t)\, 
\big(\|f\|_{L^t (\lambda )}\big)^d
\\ &\quad\text{ for } \frac{1}{t}+\frac{d-1}{2}=\frac{d+1}{2r},\quad 1\leq t<d+2 .
\end{aligned}
\end{equation}
The necessary estimate from \cite{sto} is 
\begin{equation*}
\|\lambda\ast g\|_{{\frac{d+1}{d-1}}}\leq C(n)\, \|g\|_{\frac{d^2 +d}{d^2 -d+2}},
\end{equation*}
from which it follows by the Hausdorff-Young inequality that 
\begin{equation}\label{ineq3poly}
\|\widehat{fd\lambda\ast g}\|_{\frac{d+1}{2}}\lc\ci{n}\, \|f\|_{L^{\infty}(\lambda )}\,\|g\|_{\frac{d^2 +d}{d^2 -d+2}}.
\end{equation}

Drury's argument is an iterative one. Begin by assuming that the $L^{p_0}\to L^{q_0}$ estimate \eqref{ineq1poly} holds for some $p_0$ and $q_0$
satisfying $\frac{1}{p_0}+\frac{d(d+1)}{2q_0}=1$.
We then  also have
\begin{equation}\label{ineq4poly}
\|\widehat{fd\lambda\ast g}\|_{s_0}
\leq 
C(n,p_0 ) \|f\|_{L^{p_0} (\lambda )}\,
\|g\|_2, \quad \frac{1}{s_0}=\frac{1}{q_0}+\frac{1}{2}.
\end{equation}
To see this  write $\|\widehat{fd\lambda\ast g}\|_{s_0}
=\|\widehat{fd\lambda}\,\widehat g\|_{s_0}$ and estimate this by
$\|\widehat{fd\lambda}\|_{q_0}\,\|\widehat{g}\|_2 $, using H\"older's inequality.
Now use the assumed $L^{p_0}\to L^{q_0} $ inequality and Plancherel's formula to 
get \eqref{ineq4poly}.
%
Interpolation of \eqref{ineq3poly} and \eqref{ineq4poly} gives
\begin{equation}\label{ineq6poly}
\|\widehat{fd\lambda\ast g}\|_{s} \lc\ci{n,\vth, p_0} \|f\|_{L^{a} (\lambda )}\,
\|g\|_r
\end{equation}
where
\Be\label{ineq7poly}
\Big( \frac 1s, \frac 1a,\frac 1r\Big)=(1-\vth)
\Big(\frac{1}{s_0},\frac{1}{p_0}, \frac 12\Big) +
\vth \Big( \frac{2}{d+1}, \frac{1}{\infty}, \frac{d^2-d+2}{d^2+d}\Big)
\Ee
for $0<\vth<1$.
We wish to apply this inequality with  $g$ equal to the $d$-fold convolution in \eqref{ineq2poly}.
This restricts the $r$-range to $r<\frac{d+2}{d}$ 
(corresponding to the range $t<d+2$).
A calculation shows that this restricts the range of $\vth$ in 
\eqref{ineq7poly} to 
\Be\label{thetamin}
\frac{(d-2)(d+1)d}{(d+2)(d^2-3d+4)}
=:\vth_{\text{min}} <\vth<1
\Ee
With $g= fd\la*\cdots *fd\la$ we obtain from  \eqref{ineq6poly} 
\begin{equation}\label{ineq8poly}
\|\big(\widehat{fd\lambda}\big)^{d+1}\|_{s} \lc\ci{n,\vth, p_0} 
\|f\|_{L^{a} (\lambda )}\,
\big(\|f\|_{L^t (\lambda )}\big)^d 
\end{equation}
so long as $t<d+2$, $\frac{1}{t}=\frac{d+1}{2r}-\frac{d-1}2$
and $r, a, s$ are  as in \eqref{ineq7poly} with $\vth>\vth_{\text{min}}$.
With $f=\chi_E$ this becomes 
\begin{equation*}
\|\widehat{\chi_E d\lambda}\|_{(d+1)s} \lc\ci{n,\vth, p_0} 
\,\lambda (E)^{\tfrac{1}{d+1}(\tfrac{1}{a}+\tfrac{d}{t})} ,
\end{equation*}
which gives
$$
\|\widehat{fd\lambda}\|_{q} \lc\ci{n,\vth, p_0} \,\|f\|_{L^{p,1} (\lambda )},\ 
\text{ for }\frac{1}{q} =\frac{1}{(d+1)s} \ \text{and }\ \frac{1}{p}=
\frac{1}{d+1}\big(\frac{1}{a}+\frac{d}{t}\big),
$$ where $s,a,t$ are in 
$$\Big( \frac 1s, \frac 1a,\frac 1t\Big)=(1-\vth)
\Big(\frac{1}{s_0},\frac{1}{p_0}, \frac {3-d}4\Big) +
\vth \Big( \frac{2}{d+1}, \frac{1}{\infty}, \frac 1d\Big), \quad
\vth_{\text{min}}<\vth<1\,.
$$
A little algebra shows that $p$ and $q$ satisfy 
${1}/{p}+{d(d+1)}/(2q)=1$
(for any $\vth\in (\vth_{\text{min}},1))$. 
Thus we get the  restricted strong  type version of \eqref{ineq1poly} for the exponent pair $(p,q)$.  
If  $p_1$ is the exponent $p$ corresponding to the limiting 
case $\vth_{\text{min}}$ we  obtain by real interpolation 
 the sharp $L^p(\la)\to L^q$ inequality in the open range  $1\le p<p_1$. 
Using $\vth_{\text{min}}$ in \eqref{thetamin}  we calculate that 
\begin{equation*}
\frac{1}{p_1}=\frac{8}{(d+1)(d+2)(d^2 -3d+4)}\cdot\frac{1}{p_0}+
\frac{d}{(d+1)(d+2)}.
\end{equation*}
If we define recursively 
$\frac{1}{p_{j+1}}=\frac{8}{(d+1)(d+2)(d^2 -3d+4)p_j}+\frac{d}{(d+1)(d+2)}$ 
then the sequence $\{p_0 ,p_1 ,p_2 ,\dots\}$  converges to 
$$\frac{d^3 -3d+6}{d^2 -3d+4}=d+3+\frac{2(d-3)}{d^2-3d+4}
$$
and we can conclude that  \eqref{ineq1poly} holds for $p$ in the range 
\eqref{ineq1.6poly}.
\end{proof}

\medskip


\appendix

\section{Some results from   interpolation theory} 
\label{appendixinterpol}
We gather various interpolation results used  in the paper, especially in 
\S\ref{multlinsymsect}.
They can be found more or less explicit in the literature and no originality is claimed.
In some cases  it is hard to cite exactly the precise statement that we need 
and the reader might find the inclusion of this appendix helpful.

\subsection*{On complex interpolation of multilinear operators}

We  use complex interpolation for multilinear operators defined 
for functions in a quasi-normed space with values in a Lorentz-space. 
We  limit ourselves to the statements needed in this paper
where the target space of our operator is not varied.

In the following lemma we let the measure space 
$\cM$ be a finite set,  with counting measure, and let $V$ be a Lorentz space. For a positive weight on $\cM$ the norm in 
$\ell^p(w)$ is given by $\|f\|_{\ell^p(w)}=(\sum_{x\in \cM} |f(x)|^pw(x))^{1/p}$.

\begin{lemma} \label{multilinonseq}
Let 
$T$ be a multilinear operator  defined on $n$-tuples of functions on $\cM$
and suppose that for some some weights $w_{i,0}$,
$w_{i,1}$ on $\cM$ and $p_{i,0}, p_{i,1}\in (0,\infty]$ 
\Be\begin{aligned}
&\|T[f_1,\dots, f_n]\|_V \le M_0 \prod_{i=1}^n
\|f_i\|_{\ell^{p_{i,0}}(w_{i,0})}\,,
\\
&\|T[f_1,\dots, f_n]\|_V \le M_1 \prod_{i=1}^n
\|f_i\|_{\ell^{p_{i,1}}(w_{i,1})}\,.
\end{aligned}
\Ee
Let $0<\vth<1$ and define $p_i$ and weight functions $w_i$ by
\begin{align}
\frac{1}{p_i}&=\frac{(1-\vth)}{p_{i,0}} +\frac{\vth}{p_{i,1}}\,,
 \\w_i &= [w_{i,0}^{(1-\vth)/p_{i,0}}
w_{i,1}^{\vth/p_{i,1}}]^{p_i}\,.
\end{align}
Then  there is $C$ (independent of the $f_1,\dots, f_n$ and $\cM$) so that
\Be
\|T[f_1,\dots, f_n]\|_V \le CM_0^{1-\vth}M_1^\vth \prod_{i=1}^n
\|f_i\|_{\ell^{p_{i}}(w_{i})}\,.
\Ee
\end{lemma}

\begin{proof}
This is an adaptation of the standard argument in complex interpolation (\cite{sw}) for analytic families of operators, in the setting for Lorentz spaces in \cite{sagher}. We assume that $V$ is a Lorentz space associated to the measure space $\Omega$ with measure $\mu$, say
$L^{r,q}(\mu)$. Let  $\rho\in (0,1]$ and assume in addition 
 $\rho<\min \{q,r\}$. Define 
 the maximal function $h^{**}_\rho\equiv h^{**}$ on $(0,\infty)$ by 
\Be\label{maxrearr} 
h^{**}(t)=\begin{cases}
\sup_{E: \mu(E)>t} (\mu(E)^{-1}\int_E|f(y)|^\rho d\mu(y))^{1/\rho}, \quad&
t\in (0,\mu(\Omega))
\\
 (t^{-1}\int_\Omega|f(y)|^\rho d\mu(y))^{1/\rho}, \quad&  t\in [\mu(\Omega),\infty)
\end{cases}.\Ee
 The function $[h^{**}]^\rho$ is dominated by the Hardy-Littlewood maximal function of $[h^*]^\rho$, where $h^*$ is 
 the nonincreasing rearrangement  of $h$.

For a function $g$ on $(0,\infty)$ set $\|g\|_{\la^{r,q}}= (\tfrac qr
\int_0^\infty t^{q/r} |g(t)|^q \tfrac{dt}{t})^{1/q}$ if $q<\infty$ and 
$\|g\|_{\la^{r,\infty }}= \sup_{t>0} t^{1/r} |g(t)|$.
Then  Hunt \cite{hu} showed that 
the expression $\|h^{**}\|_{\la^{r,q}}$ is  a quasi-norm on $L^{q,r}$ which makes 
$L^{q,r}$ 
a $\rho$-convex space.

Let $S$ be the strip $\{z=\vth+\ic\tau: 0<\vth<1, \,\tau\in \bbR\}$ and 
$\overline S$ its closure.
Let $f_i\in \ell^{p_i}(w_i)$ so that $\|f_i\|_{\ell^{p_i}(w_i)}\le 1$
 and define for $x\in \cM$
$$f_{i,z}(x) = 
e^{\ic\arg(f(x))}
 \frac{[|f(x)|^{p_i}w_i(x)]^{\frac{1-z}{p_{i,0}}+\frac{z}{p_{i,1}}}}
{w_{i,0}(x)^{\frac{1-z}{p_{i,0}}} w_{i,1}(x)^{\frac{z}{p_{i,1}}}}\,.
$$
Then $f_{i,\vth}=f$. Moreover, 
$\|f_z\|_{\ell^{p_{i,0}}(w_{i,0})}= \|f\|_{\ell^{p_i}(w_i)}^{p_{i,0}/{p_i}}$
if $\Re(z)=0$, and
$\|f_z\|_{\ell^{p_{i,1}}(w_{i,1})}= \|f\|_{\ell^{p_i}(w_i)}^{p_{i,1}/{p_i}}$
if $\Re(z)=1$.
We define, for $y\in\Omega$,
$$H_z(y)=
T[f_{1,z},\dots, f_{1,z}](y)\,.
$$
Then $H_\vth= T[f_{1},\dots, f_{1}]$ and we must show that $\|H_\vth\|_V\lc M_0^{1-\vth}M_1^\vth$.
For almost every $y\in \Omega$ the function $z\mapsto H_z(y)$ is bounded and 
analytic in $S$, continuous on $\overline S$.

We use the standard properties of the Poisson-kernel associated with $S$,
 see Ch. V.4 in \cite{sw}. Let 
$P_0(\vth,t)=\frac 12 \frac{\sin(\pi\vth)}{\cosh (\pi  t)-\cos(\pi\vth)}$,
$P_1(\vth,t)=\frac 12 \frac{\sin(\pi\vth)}{\cosh (\pi t)+\cos(\pi\vth)}$.
For $0\le\vth\le 1$ we then have 
$\int_{-\infty}^\infty P_0(\vth,t)dt=(1-\vth)$,
$\int_{-\infty}^\infty P_1(\vth,t)dt=\vth$. 
Thus, proceeding exactly as in \cite{sagher}  we have
$$
\,\log|H_{\vth+\ic \tau}(y)
|\le \int_{-\infty}^\infty P_0(\vth,\tau)\log|H_{\ic \tau}(y)|d\tau +
\int_{-\infty}^\infty P_1(\vth,\tau)\log|H_{1+\ic \tau}(y)|d\tau
$$
and then
$$\begin{aligned}
|H_{\vth+\ic \tau}(y)
|\le\, &\Big(\exp\Big(\frac{1}{1-\vth}
\int_{-\infty}^\infty P_0(\vth,\tau)\log[|H_{\ic \tau}(y)|^\rho]  d\tau\Big)\Big)^{\frac{1-\vth}{\rho}}\\&
\times
\Big(\exp\Big(\frac{1}{\vth}
\int_{-\infty}^\infty P_1(\vth,\tau)\log[|H_{1+\ic \tau}(y)|^\rho]  d\tau\Big)\Big)^{\frac{\vth}{\rho}}\,.
\end{aligned}
$$
By Jensen's inequality,
$$
|H_{\vth+\ic \tau}(y)|\le A_0(y)^{1-\vth} A_1(y)^{\vth}$$ where
$$\begin{aligned}
A_0(y)&=  \Big(\frac{1}{1-\vth}
\int_{-\infty}^\infty P_0(\vth,\tau)|H_{\ic \tau}(y)|^\rho  d\tau\Big)^{1/\rho}
\\ A_1(y)&=
\Big(\frac{1}{\vth}
\int_{-\infty}^\infty P_1(\vth,\tau)|H_{1+\ic \tau}(y)|^\rho d\tau\Big)^{1/\rho}.
\end{aligned}
$$
By H\"older's inequality   we have $( A_0^{1-\vth} A_1^{\vth})^{**}(t)\le
(A_0^{**}(t))^{1-\vth} (A_1^{**}(t))^\vth$.
By Fubini's theorem we get $A_0^{**}(t)\le B_0(t)$,
$A_1^{**}(t)\le B_1(t)$ where
$$\begin{aligned}
B_0(t)&=  \Big(\frac{1}{1-\vth}
\int_{-\infty}^\infty P_0(\vth,\tau)|H_{\ic \tau}^{**}(t)|^\rho  d\tau\Big)^{1/\rho}\,,
\\ B_1(t)&=
\Big(\frac{1}{\vth}
\int_{-\infty}^\infty P_1(\vth,\tau)|H_{1+\ic \tau}^{**}(t)|^\rho d\tau\Big)^{1/\rho}\,.
\end{aligned}
$$
Hence
$$|H_{\vth+\ic \tau}^{**}(t)|\le 
(A_0^{**}(t))^{1-\vth} (A_1^{**}(t))^\vth \le B_0^{1-\vartheta}(t)
 B_1^{\vartheta}(t)$$
and another  application of  H\"older's inequality  yields 
$$\|H_{\vth+\ic\tau}^{**}\|_{\la^{q,r}}\le 
\|B_0\|_{\la^{q,r}}^{1-\vth} 
 \|B_1\|_{\la^{q,r}}^{\vth}\,. $$
Since $q>\rho$ we can apply the integral Minkowski inequality
(as a  version of the triangle inequality in $L^{q/\rho}(0,\infty)$)
$$\begin{aligned}
\|B_0\|_{\la^{q,r}}&\le \Big(\frac{1}{1-\vth}
\int_{-\infty}^\infty P_0(\vth,\tau)
\|H_{\ic \tau}^{**}\|_{\la^{r,q}}^\rho  d\tau\Big)^{1/\rho}\,,
\\
\|B_1\|_{\la^{q,r}}&\le \Big(\frac{1}{\vth}
\int_{-\infty}^\infty P_1(\vth,\tau)\|H_{1+\ic \tau}^{**}\|_{\la^{r,q}}^\rho  d\tau\Big)^{1/\rho}\,.
\end{aligned}
$$
By assumption
$$\begin{aligned}
\|H_{\ic \tau}^{**}\|_{\la^{r,q}}&\le C M_0
 \prod_{i=1}^n
\|f_{i,\ic\tau}\|_{\ell^{p_{i,0}}(w_{i,0})} \le CM_0\,,
\\
\|H_{1+\ic \tau}^{**}\|_{\la^{r,q}}&\le C M_1
 \prod_{i=1}^n
\|f_{i,1+\ic\tau}\|_{\ell^{p_{i,1}}(w_{i,1})} \le CM_1\,.
\end{aligned}
$$
We get 
$\|H_{\vth+\ic\tau}^{**}\|_{\la^{q,r}}\le C M_0^{1-\vth}M_1^\vth$, using
 the above formulas for the integrals of $P_0$ and $P_1$.
\end{proof}

We  now use Lemma \ref{multilinonseq} and a straightforward  {\it  transference method} to prove an interplation theorem for sequences 
with values in certain  real interpolation spaces $\overline X_{\theta,q}$. 
When applying the complex interpolation method
 the 
case  $q=\infty$ may pose some difficulties which can be avoided if 
$\ell^\infty$   is replaced by the closed subspace $c_0$. In particular 
it is convenient to replace 
the real interpolation space  $\overline X_{\vth,\infty}$
by $\overline X_{\vth,\infty}^0$, the closure 
of $X_0\cap X_1$ in $\overline X_{\vth,\infty}$.
To deal with this distinction we introduce some notation for the 
following lemma.  We will work with a couple $\overline X=(X_0, X_1)$ of
  compatible complete quasi-normed spaces. If 
$q<\infty$ we denote by  $\cZ(q,s,\theta)= \ell^q_s(\overline X_{\theta,q})$  the space of $\overline X_{\theta,q}$-valued 
sequences $F=\{F_k\}_{k\in \bbZ} $ with norm
$$\|F\|_{\cZ(q,s,\theta)}= \Big(\sum_{k\in \bbZ} 2^{ks q}
\|F_k\|_{\overline X_{\theta,q}}^q\Big)^{1/q}\,.$$ 
 For $q=\infty$ we define   $\cZ(\infty,s,\theta)= 
c_0(\overline X_{\theta,\infty}^0)$, a  closed subspace of 
$\ell^\infty_s(\overline X_{\theta,\infty})$,
 with norm
$\sup_{k\in \bbZ} 2^{ks}\|F_k\|_{\overline X_{\theta,\infty}}$.
We shall say that
$F=\{F_k\}\in \cZ(q,s,\theta)$ is compactly supported if $F_k=0$ except for finitely many $k$.
\begin{lemma} \label{transferlemma}
For $i=1,\dots, n$, let  $0<\thiz, \thio<1$, 
$0<\qiz,\qio\le \infty$, $\sio, \siz\in \bbR$.
Let $T$ be an $n$-linear operator defined a priori on 
$n$-tuples of compactly  supported $(X_0\cap X_1)$-valued sequences, 
with values in a Lorentz space $V$, 
and suppose that for such sequences the inequalities 
\Be\label{Zassumpt}
\|T[F_{1},\dots, F_{n}]\|_V \le \begin{cases}  M_0\prod_{i=1}^n
\|F_i\|_{\cZ(\qiz,\siz,\thiz)}
\\ M_1 \prod_{i=1}^n
\|F_i\|_{\cZ(\qio,\sio,\thio)}
\end{cases}
\Ee 
hold. Define  $q_{i}$, $s_{i}$ and $\theta_{i}$ by
$$\Big(\frac {1}{q_{i}}, s_{i}, \theta_{i}\Big)=
(1-\vth)
\Big(\frac {1}{q_{i,0}}, s_{i,0}, \theta_{i,0}\Big)+\vth
\Big(\frac {1}{q_{i,1}}, s_{i,1}, \theta_{i,1}\Big)\,.
$$
Then  $T$ uniquely extends to an operator bounded on 
$\prod_{i=1}^n \cZ(q_{i}, s_{i},\theta_{i})$ so that
$$\|T[F_{1},\dots, F_{n}]\|_V \lc  M_0^{1-\vth}M_1^\vth \prod_{i=1}^n
\|F_i\|_{\cZ(q_i,s_i,\theta_i)}\,.$$
\end{lemma}

\begin{proof}
The uniqueness of the extension is clear because of the density of compactly 
supported  $X_0\cap X_1$-valued functions in $\cZ(q,s,\theta)$ (for $q=\infty$ this requires the modification in the definition using 
$c_0(\overline X_{\theta,\infty}^0)$). 
In what follows we write matters out for the case that the $q_i<\infty$ and leave the obvious notational modifications in the case $q_i=\infty$ to the reader.
It will be convenient to use the characterization of $\overline X_{\theta,q}$ by means of the $J$-functional.

It suffices to prove  that for   $\|F_i\|_{\cZ(q_{i},s_{i},\theta_{i})}\le 1$, 
$i=1,\dots,d$,
\Be\label{goalTbd}
\|T(F_1,\dots, F_n)\|_V \lc M_0^{1-\vth} M_1^\vth .
\Ee
We 
 write $F_i=\{F_{i,k}\}$ with $F_{i,k} \in \overline X_{\th_{i},q_i}$ and
$$ \Big(\sum_k \big[ 2^{k s_{i}}\|F_{i,k}\|_{\overline X_{\th_{i} ,q_{i},J}} \big]^{q_{i}} \Big)^{1/q_{i}} \le 1 .$$
We can decompose $F_{i,k}=\sum_l u_{i,k,l}$ with $u_{i,k,l}\in X_0\cap X_1$ and convergence in $X_0+X_1$ so that
\[
\Big(\sum_l\big[ 2^{-l\th_{i} } J(2^l, u_{i,k,l};\overline
X)\big]^{q_{i}} \Big)^{1/q_{i}}\le (1+2^{-|k|-2})
\|F_{i,k}\|_{\overline X_{\th_{i} ,q_{i}; J}}
\]
and thus $\big(\sum_{k,l} 
[ 2^{k s_{i}} 2^{-l \th_{i}} J(2^l, u_{i,k,l};\overline X) 
]^{q_{i}}\Big)^{1/q_i}\le 2 .$
We set $N_0=0$ and define numbers $N_0<N_1<N_2< ...$ so that 
\Be\label{defNnu}
\Big(\sum_{\substack{ k,l\\ \max\{|k|,|l|\}\ge N_\nu }} \big[ 2^{k s_{i}} 2^{-l \th_{i}}
J(2^l, u_{i,k,l};\overline X) \big]^{q_{i}}\Big)^{1/q_{i}}\le 2^{-\nu} .
\Ee
for $i=1,\dots, n$.
Let $\chi_{\nu}(k,l)=1$ if $N_{\nu}\le \max\{|k|,|l|\}<N_{\nu+1}$ and 
$\chi_{\nu}(k,l)=0$ otherwise, and let  $F_{i,k}^{\nu} =\sum_l \chi_{\nu}(k,l) u_{i,k,l}$. Then $\sum_{\nu=1}^\infty F_{i}^{\nu}=F_i$ 
with convergence in $\cZ(q_{i}, s_{i}, \theta_{i})$ and, for each $k$, 
$\sum_{\nu=1}^\infty F_{i,k}^{\nu}=F_{i,k}$  with convergence in 
$\overline X_{\theta_{i},q_{i}} $ and {\it a fortiori} with convergence in 
$X_0+X_1$.

In order to prove \eqref{goalTbd} 
we fix the chosen vectors  $u_{i,k,l}$ and  define  operators acting on 
$n$-tuples of functions $\fa=\{\fa_{k,l}\}$ defined on a subset of  $\bbZ\times
\bbZ$.  Let $\cM_{\vec\nu}=\{(k,l):\max\{k,l\} \le 
\max\{ N_{\nu_1+1},\dots, 
N_{\nu_n+1}\}\}$. 
For any $n$-tuple $\vec\nu=(\nu_1,\dots, \nu_n)$ of nonnegative integers we let
$\cF^{\nu_i}_i(\fa)= \{\cF^{\nu_i}_{i,k}(\fa)\}_{k\in\bbZ}$, where for
$i=1,\dots,n$
$$\cF^{\nu_i}_{i,k}(\fa)=\sum_{l}
\chi_{\nu_i}(k,l)\fa_{k,l}\, u_{i,k,l} .$$ Now define for an $n$-tuple  of such sequences  a multilinear operator $S_{\vec\nu}$ by
$$
S_{\vec\nu}(\fa^1,\dots, \fa^n) = T[\cF^{\nu_1}_1(\fa^1), \dots, \cF^{\nu_n}_n(\fa^n)] .
$$
Let $$w_{q,s,\th}(k,l)= \big[2^{ks}2^{-l\theta}
(\eps+J(2^l,u_{i,k,l};\overline X))\big]^q \,.$$
It is our objective to show 
\begin{equation}\label{Sestimate}\|S_{\vec\nu}[\fa^1,\dots, \fa^n] \|_V \le C
M_0^{1-\vth}M_1^\vth\prod_{i=1}^n \|\fa^i\|_{\ell^{q_i}(w_{q_i,s_i,\th_i})}
\end{equation}
where the constant $C$ is independent of the
choice of the specific $u_{i,k,l}$ and independent of $\vec \nu$.
The inclusion of $\eps$ in the definition of $w_{q,s,\theta}$ guarantees the positivity of the weight.
Once the bound \eqref{Sestimate} is verified we will
then apply it  to the sequences
 $\fa^i_{k,l}=\chi_{\nu_i}(k,l)$. For this  choice of the $\fa^i$  an estimate  for the expression 
$\|S_{\vec\nu}[\fa^1,\dots, \fa^n]\|_V$ 
becomes an estimate for $\|T[F^{\nu_1}_1,\dots, F^{\nu_n}_n]\|_V$, 
after letting $\eps\to 0$.

Now for any admissible choice of $q,s,\theta$
$$\|\cF_i^{\nu_i}(\fa^i)\|_{\cZ(q,s,\theta)}=  \Big(\sum_k \Big[ 2^{ks} \Big\| 
\sum_l \chi_{\nu_i}(k,l)
\fa_{k,l}^i \, u_{i,k,l} \Big\|_{\overline X_{\vth ,q;J}}
\Big]^{q} \Big)^{1/q}$$ and by definition of the 
$\overline X_{\vth ,q;J}$ norm we have
$$\Big\| 
\sum_l \chi_{\nu_i}(k,l)
\fa_{k,l}^i \, u_{i,k,l} \Big\|_{\overline X_{\vth ,q;J}}
\le 
\Big(\sum_l \big[ 2^{-l\vth} 
J(2^l,  \fa_{k,l}^i \, u_{i,k,l};\overline X)
\big]^{q}
\Big)^{1/q}\,.
$$
Therefore, by the homogeneity of the $J$-functional 
$$
\|\cF_i^{\nu_i}(\fa^i)\|_{\cZ(q,s,\theta)}\lc
\Big(\sum_{k,l} 2^{(ks-l\theta)q}[
\chi_{\nu_i}(k,l)J(2^l, u_{i,k,l};\overline X)|\fa^i_{k,l}|]^q\Big)^{1/q}.
$$
Notice that we can consider the sequences $\fa^i$ as functions defined on 
$\cM_{\vec\nu}$. 
By assumption 
\begin{multline*}
\| T[\cF_1(\fa^1), \dots, \cF_n(\fa^n)]\|_V\\
\lc\min\big\{
M_0\prod_{i=1}^n \|\cF_i(\fa^i)\|_{\cZ(\qiz,\siz,\thiz)},\, 
M_1\prod_{i=1}^n \|\cF_i(\fa^i)\|_{\cZ(\qio,\sio,\thio)}\big\},
\end{multline*}
 and by the above this implies 
\begin{multline*}\|S_{\vec\nu}[\fa^1,\dots, \fa^n]\|_V\\
\lc \min\big\{
M_0\prod_{i=1}^n\|\fa^i\|_{\ell^{\qiz}(w_{\qiz,\siz,\thiz})},\,
M_1\prod_{i=1}^n\|\fa^i\|_{\ell^{\qio}(w_{\qio,\sio,\thio})}\big\}.
\end{multline*}
Since $w_{q_i,s_i,\th_i}= [w_{\qiz,\siz,\thiz}^{(1-\vth)/\qiz}
w_{\qio,\sio,\thio}^{(1-\vth)/\qio}]^{q_i}$,  Lemma 
\ref{multilinonseq} now gives   \eqref{Sestimate}, with a 
constant independent of $\vec \nu$.

Finally if we apply \eqref{Sestimate} with the sequences
 $\fa^i_{k,l}=\chi_{\nu_i}(k,l)$  and let $\eps\to0$  we obtain
\begin{multline*} \|T[F^{\nu_1}_1,\dots, F^{\nu_n}_n]\|_V \\ \lc M_0^{1-\vth}M_1^\vth
\prod_{i=1}^n \Big(\sum_{k,l} \big[ \chi_{\nu_i}(k,l)2^{ks_{i}} 2^{-l\theta_{i}} J(2^l,u_{i,k,l};\overline X)
\big]^{q_{i}}\Big)^{1/q_{i}}
\end{multline*}
and, by \eqref{defNnu} this expression  is  $\lc M_0^{1-\vth}M_1^\vth
 2^{-(\nu_1+\dots+\nu_n)}$.
Since $\sum_{\nu_i} F^{\nu_i}_i=F_i$ with convergence in 
$\ell^{q_{i}}(\overline  X_{\theta_{i}, q_{i}})$ we see that 
$\sum_{\vec\nu} T[F^{\nu_1}_1,\dots, F^{\nu_n}_n]$ converges in $V$ to 
$T[F_1,\dots, F_n]$ so that $\|T(F_1,\dots, F_n)\|_V\lc M_0^{1-\vth}M_1^\vth $.
\end{proof} 

\subsection*{Means}
Often one  generates new estimates 
by taking  means of given estimates. The new bounds 
 may then be interpreted as estimates on intermediate  spaces:

\begin{lemma}\label{thetar} Let $0<r\le 1$ and let $V$ be an  $r$-convex space. For $i=1,\dots, n$ let $\overline X^i=(X^i_0, X^i_1)$ be couples of
compatible quasi-normed
spaces and let $T$ be an $n$-linear operator defined on
$\prod_{i=1}^n (X^i_0\cap X^i_1)$ with values in $V$.
Suppose that
$$ \| T(f_1,\dots, f_n)\|_V
\le  \prod_{i=1}^n \| f_i\|_{X^i_0}^{1-\theta_i}
 \| f_i\|_{X^i_1}^{\theta_i} $$
for some $0< \theta_i < 1$. 
 Then there is $C>0$
so that for all  $(f_1,\dots, f_n) \in \prod_{i=1}^n X^i_0\cap X^i_1$
\begin{equation}\label{rineq}
\| T(f_1,\dots, f_n)\|_V  \le C \prod_{i=1}^n \| f_i\|_{\overline {X}^i_{\theta_i, r}}
\end{equation}
and $T$ extends to a bounded operator on
$ \prod_{i=1}^n\overline {X}^i_{\theta_i, r}$.
\end{lemma}

\begin{proof}
Let $f_i \in \overline X_{\theta_i,r}$ and thus 
$f_i=\sum_{l_i\in \bbN} u_{i,l}$ 
with $u_{i,l}\in X_0^i\cap X_1^i$ and convergence in 
$X_0^i+X_1^i$.
Since $V$ is $r$-convex one can show 
 $$\| T[f_1,\dots, f_n]\|_V^r \le C^r \sum_{\vec l\in \bbN^n} 
\| T[u_{1,l_1},\dots, u_{n,l_n}]\|_V^r \,;
$$
this follows easily by considering finite sums and a limiting argument.
By assumption and the  definition of the $J$-functional the right hand side of the last display is dominated by $C^r$ times
\begin{align*}
\sum_{\vec l\in \bbN^n} \prod_{i=1}^n \big[
&\|u_{i,l_i}\|_{X_{i,0}}^{1-\theta_i} \| u_{i, l_i}\|_{X_{i,1}}^{\theta_i} \big]^r \\
\, &\le\,  \sum_{\vec l\in \bbN^n} \prod_{i=1}^n \big[
J(2^{l_i}, u_{i,l_i};\overline X_i)^{1-\theta_i}
(2^{-l_i} J(2^{l_i}, u_{i,l_i};\overline X_i))^{\theta_i} \big]^r
\\&=
\prod_{i=1}^n \Big(\sum_{l_i\in \bbN} \big[ 2^{-l_i\theta} 
J(2^{l_i}, u_{i,l_i};\overline X_i)\big]^r\Big)\,.
\end{align*}
Taking the $r$-th roots and then the infimum over all
decompositions $\{u_{i,l}\}$ of $f_i$, we get assertion
\eqref{rineq} (by the
equivalence of the $J$- and $K$-methods). The operator $T$ 
extends to $ \prod_{i=1}^n\overline {X}^i_{\theta_i, r}$ since $X_0^i\cap X_1^i$ is dense in
$\overline {X}^i_{\theta_i, r}$.
\end{proof}

\subsection*{Spaces of vector-valued sequences}
We use two results on interpolation of  $\ell^p_s(X)$ spaces, 
 for  quasi-normed  $X$ 
and $0<p\le\infty$, $s\in \bbR$.
For fixed $X$ the following standard 
formula for the real interpolation spaces can be found in \S 5.6 of \cite{BL}.
\Be\label{seqinterpolfixedA}
\begin{aligned}\big(\ell_{s_0}^{q_0}(X),\ell_{s_1}^{q_1}(X)\big)_{\vth,q}= 
\ell_s^q(X), \quad s=(1-\vth)s_0+\vth s_1&
\\
\qquad\text{provided that } \, s_0\neq s_1,\,\,0<q_0\le\infty, \,\, 0<q_1\le\infty\,.&
\end{aligned}
\Ee

Next consider the space $(\ell^r_{s_0} (X_0), \ell^r_{s_1} (X_1))_{\vth, q}$ 
for a 
pair of compatible  quasi-normed spaces $(X_0, X_1)$.
The following lemma is essentially in Cwikel's paper \cite{Cw} 
who considered  normed spaces.  We include a proof for the convenience of the reader.

\begin{lemma}\label{quasi}
Suppose that $X_0$ and $X_1$ are compatible quasi-normed spaces.
Let $0<r\le \infty$, $s_0$, $s_1 \in \bbR$, and $0<\vth <1$. If $
r\le q\le \infty$, then there is the continuous embedding
\begin{equation}\label{obvious}
\ell^r_s ( (X_0, X_1)_{\vth, q})\hookrightarrow 
(\ell^r_{s_0} (X_0), \ell^r_{s_1} (X_1))_{\vth, q}, \quad
s= (1-\vth)s_0 + \vth s_1\,.
\end{equation}
\end{lemma}

{\it Remark.} When $0< q \le r$, a slight modification of the argument sketched
below shows that the inclusion in \eqref{obvious} reverses
direction. In particular, equality holds when $q=r$. But this fact
is not needed here. 
Examples disproving the  equality in the cases $q\neq r$ are in 
\cite{Cw}.

\begin{proof}[Proof of Lemma \ref{quasi}]
Let $f=\{f_k\} \in \ell^r_{s_0} (X_0)+ \ell^r_{s_1} (X_1)$. Fix
$t>0$ and $\varepsilon>0$. For each $k \in \bbZ$, choose $f_{0,k}$
and $f_{1,k}$ with $f_k = f_{0,k} + f_{1,k}$ such that
$$ \| f_{0,k}\|_{X_0} + 2^{k(s_1-s_0)} t \, \| f_{1,k}\|_{X_1} \le
(1+\varepsilon) K(2^{k(s_1-s_0)} t, f_k; \overline X) .$$

Let $W_0 = \ell^r_{s_0} (X_0)$,
$W_1 = \ell^r_{s_1} (X_1)$, and let 
${K}(t,f; \overline W)$
 be  the
$K$-functional for the pair $(W_0, W_1)$. Then %
\begin{align*}
K(t,f; \overline W) &\le \Big( \sum_k [2^{ks_0} \|
f_{0,k}\|_{X_0}]^r \Big)^{1/r} + t \, \Big( \sum_k [2^{ks_1} \|
f_{1,k}\|_{X_1}]^r \Big)^{1/r} \\
& \approx \Big( \sum_k \big[2^{ks_0} \| f_{0,k}\|_{X_0} + 2^{ks_1} t
 \| f_{1,k}\|_{X_1})\big]^r \Big)^{1/r}\\
&\le (1+\varepsilon) \Big( \sum_k \big[2^{ks_0} K(2^{k(s_1-s_0)} t,
f_k; \overline X)]^r \Big)^{1/r} .
\end{align*}

If $r\le q<\infty$, then it follows by Minkowski's inequality that
\begin{align*} \| f\|_{\overline W_{\vth, q} }
&= \Big( \int_0^{\infty} \big[t^{-\vth} {K}(t,f;\overline W) \big]^q
\frac{dt}{t}
\Big)^{1/q}\\
&\lc \Big( \int_0^{\infty} t^{-\vth q} \Big( \sum_k [2^{ks_0}
K(2^{k(s_1-s_0)} t, f_k; \overline X)]^r \Big)^{q/r} {dt\over
t}\Big)^{1/q}\\
& \le \Big( \sum_k \Big[ \int_0^{\infty} \Big( t^{-\vth }\,2^{ks_0} K(2^{k(s_1-s_0)} t, f_k; \overline X) \Big)^{q} {dt\over
t}\Big]^{r/q} \Big)^{1/r} .
\end{align*}
Let  $s= (1-\vth)s_0 + \vth s_1$. 
By the change of variables $u =2^{k(s_1-s_0)} t$, we see that the right hand side 
of the last display equals a constant multiple of
\begin{align*}
&\Big( \sum_k \Big[ 2^{ks} \Big(\int_0^{\infty} \Big( u^{-\vth }
 K(u, f_k; \overline X) \Big)^{q} {du\over
u}\Big)^{1/q}\Big]^r \Big)^{1/r} \\
& = \Big( \sum_k \big[ 2^{ks} \| f_k\|_{\overline X_{\vth,q}} \big]^r
\Big)^{1/r}  = \| f\|_{\ell^r_s (\overline X_{\vth,q})}\,.
\end{align*}
The case $q=\infty$ is similar.
\end{proof}

\medskip

\noindent{\it Acknowledgement.}
The first named author would like to thank Sanghyuk Lee for several useful conversations about the subject matter. We also thank the referee for valuable 
comments.

\end{document}